\documentclass[a4paper,10pt]{article}
\usepackage{a4wide}
\usepackage{amsmath,amssymb}
\usepackage[english]{babel}
\usepackage{multirow}
\usepackage{verbatim}
\usepackage{float}
\usepackage[section]{placeins}
\usepackage{graphicx}
\usepackage{pdfpages}
\usepackage{xargs}
\usepackage{mathtools}
\usepackage{dsfont}                     
\usepackage{textcomp}
\usepackage[colorinlistoftodos,prependcaption,textsize=tiny]{todonotes}
\usepackage{pgfplots}
\pgfplotsset{compat = 1.14}
\usepackage{tikz-network}
\usetikzlibrary{positioning}
\usepackage[title]{appendix}
\usepackage{booktabs}
\usepackage{cite}
\usepackage{subcaption}

\usepackage{pifont}

\DeclarePairedDelimiter{\floor}{\lfloor}{\rfloor}

\usepackage{parskip}
\definecolor{TomColBack}{HTML}{F62217}
\definecolor{TomCol}{HTML}{16EAF5}

\usepackage{fancyhdr} 
\newcommand\smallO{
  \mathchoice
    {{\scriptstyle\mathcal{O}}}
    {{\scriptstyle\mathcal{O}}}
    {{\scriptscriptstyle\mathcal{O}}}
    {\scalebox{.6}{$\scriptscriptstyle\mathcal{O}$}}
  }

\usepackage{tikz}
\usetikzlibrary{arrows}
\usetikzlibrary{positioning}
\usetikzlibrary{decorations}

\definecolor{mylilas}{HTML}{CC0099}
\usepackage{listings}
\definecolor{mygreen}{rgb}{0,0.6,0} 
\definecolor{mybackground}{HTML}{F7BE81}
\usepackage{caption}
\DeclareCaptionFont{white}{\color{white}}
\DeclareCaptionFormat{listing}{\colorbox{gray}{\parbox{\textwidth}{#1#2#3}}}
\usepackage{hyperref}
\hypersetup{
    pagebackref=true,
    colorlinks=true,
    citecolor=red,
    linkcolor=blue,
    linktoc=page,
    urlcolor=blue,
    pdfauthor= Mark Christianen,
    pdftitle=Stability-EV
}
\usepackage{aligned-overset}
\definecolor{redText}{rgb}{1,0,0}

\definecolor{blueText}{HTML}{0080FF}

\definecolor{greenText}{HTML}{00e600}

\usetikzlibrary{chains,shapes.multipart}
\usetikzlibrary{shapes,calc}
\usetikzlibrary{automata,positioning}
\usepackage{cases}

\definecolor{myred}{RGB}{220,43,25}
\definecolor{mygreen}{RGB}{0,146,64}
\definecolor{myblue}{RGB}{0,143,224}

\definecolor{MarkCol}{HTML}{B216FA}
\definecolor{MarkColBack}{HTML}{5FFB17}
\newcommandx{\Mark}[2][1=]{\todo[linecolor = MarkCol, backgroundcolor = MarkColBack!30, bordercolor = MarkColBack,#1]{#2}}

\tikzset{
myshape/.style={
  rectangle split,
  minimum height=1.5cm,
  rectangle split horizontal,
  rectangle split parts=8, 
  draw, 
  anchor=center,
  },
mytri/.style={
  draw,
  shape=isosceles triangle,
  isosceles triangle apex angle=60,
  inner xsep=0.65cm
  }
}

\usepackage{algorithm}
\usepackage[noend]{algpseudocode}
\usepackage{amsthm}
\newtheorem{theorem}{Theorem}[section]
  \usepackage{titling}
\predate{}
\postdate{}
\newtheorem{lemma}{Lemma}[section]

\newtheorem{corollary}{Corollary}[section]

\newtheorem{remark}{Remark}[section]
\newtheorem{assumption}{Assumption}[section]
\DeclareMathOperator*{\argmax}{arg\,max}

\usepackage{color, colortbl}
\usepackage{authblk}
\numberwithin{equation}{section}
\usepackage{enumitem}

\def\@adminfootnotes{%
  \let\@makefnmark\relax  \let\@thefnmark\relax
  \ifx\@empty\@date\else \@footnotetext{\@setdate}\fi
  \ifx\@empty\@subjclass\else \@footnotetext{\@setsubjclass}\fi
  \ifx\@empty\@keywords\else \@footnotetext{\@setkeywords}\fi
  \ifx\@empty\thankses\else \@footnotetext{%
    \def\par{\let\par\@par}\@setthanks}%
  \fi
}

\begin{document}

\author[1]{\small Christianen, M.H.M.}
\author[1]{\small Janssen, A.J.E.M.}
\author[1,2]{\small Vlasiou, M.}
\author[1,3]{\small Zwart, B.}

\affil[1]{\footnotesize Eindhoven University of Technology}
\affil[2]{\footnotesize University of Twente}
\affil[3]{\footnotesize Centrum Wiskunde \& Informatica}

\title{Asymptotic analysis of Emden-Fowler type equation with an application to power flow models}
\date{}
\maketitle

\begin{abstract}
Emden-Fowler type equations are nonlinear differential equations that appear in many fields such as mathematical physics, astrophysics and chemistry. In this paper, we perform an asymptotic analysis of a specific 
Emden-Fowler type equation that emerges in a queuing theory context as an approximation of voltages under a well-known power flow model. Thus, we place Emden-Fowler type equations in the context of electrical engineering. We derive properties of the continuous solution of this specific Emden-Fowler type equation and study the asymptotic behavior of its discrete analog. We conclude that the discrete analog has the same asymptotic behavior as the classical continuous Emden-Fowler type equation that we consider.
\end{abstract}




\section{Introduction}

Many problems in mathematical physics, astrophysics and chemistry can be modeled by an Emden-Fowler type equation of the form
\begin{align}
\frac{d}{dt}\left(t^{\rho}\frac{du}{dt} \right)\pm t^{\sigma}h(u) = 0,\label{eq:general_fowler_emden}
\end{align} where $\rho,\sigma$ are real numbers, the function $u:\mathbb{R}\to\mathbb{R}$ is twice differentiable and $h: \mathbb{R}\to\mathbb{R}$ is some given function of $u$. For example, choosing $h(u)=u^n$ for $n\in\mathbb{R}$, $\rho=1$, $\sigma=0$ and plus sign in \eqref{eq:general_fowler_emden}, is an important equation in the study of thermal behavior of a spherical cloud of gas acting under the mutual attraction of its molecules and subject to the classical laws of thermodynamics \cite{Bellman1953, Davis}. Another example is known as \emph{Liouville's equation}, which has been studied extensively in mathematics \cite{Dubrovin1985}. This equation can be reduced to an Emden-Fowler type equation with $h(u)=e^u$, $\rho = 1,\sigma=0$ and plus sign \cite{Davis}. For more information on different applications of Emden-Fowler type equations, we refer the reader to \cite{Wong1975}.



In this paper, we study the Emden-Fowler type equation where $h(u) = u^{-1}$, $\rho = 0$, $\sigma = 0$, with the minus sign in \eqref{eq:general_fowler_emden}, and initial conditions $u(0)=k^{-1/2}, u'(0)=k^{-1/2}w$ for $w\geq 0$. For a positive constant $k>0$, we consider the change of variables $u=k^{-1/2}f$, with resulting equation
\begin{align}
\frac{d^2f}{dt^2} = \frac{k}{f},\quad t\geq 0; \quad f(0)=1,f'(0)=w.\label{eq:voltages_approx}
\end{align}

This specific Emden-Fowler type equation \eqref{eq:voltages_approx} arises in a queuing model \cite{Christianen2021}, modeling the queue of consumers (e.g.\ electric vehicles (EVs)) connected to the power grid. The distribution of electric power to consumers leads to a resource allocation problem which must be solved subject to a constraint on the voltages in the network. These voltages are modeled by a power flow model 
known as the Distflow model; see Section \ref{subsec:background_voltages} for background. The Distflow model equations are given by a discrete version of the nonlinear differential equation \eqref{eq:voltages_approx} and can be described as
\begin{align}
V_{j+1}-2V_j+V_{j-1} = \frac{k}{V_j},\quad j=1,2,\ldots; \quad V_0 = 1, V_1 = 1+k.\label{eq:voltages_distflow}
\end{align} 



In this paper, we study the asymptotic behavior and associated properties of the solution of \eqref{eq:voltages_approx} using differential and integral calculus, and show its numerical validation, 
i.e., we show that the solutions of \eqref{eq:voltages_approx}
have asymptotic behavior
\begin{align}
f(t)\sim t\left(2k\ln(t)\right)^{1/2},\quad t\to\infty,\label{eq:continuous_asympt_behavior}
\end{align} which can be used in the study of any of the aforementioned resource allocation problems. It is natural to expect that the discrete version  \eqref{eq:voltages_distflow} of the Emden-Fowler type equation has the asymptotic behavior of the form \eqref{eq:continuous_asympt_behavior} as well. However, to show \eqref{eq:discrete_asympt_behavior} below, is considerably more challenging than in the continuous case, and this is the main technical challenge addressed in this work.
We show the asymptotic behavior of the discrete recursion, as in \eqref{eq:voltages_distflow} to be
\begin{align}
V_j \sim j\left(2k\ln(j)\right)^{1/2},\quad j\to\infty.\label{eq:discrete_asympt_behavior}
\end{align}\\
There is a huge number of papers that deal with various properties of solutions of Emden-Fowler differential equations \eqref{eq:general_fowler_emden} and especially in the case where $h(u)=u^n$ or $h(u)=\exp(nu)$ for $n\geq 0$. In this setting, for the asymptotic properties of solutions of an Emden-Fowler equation, we refer to \cite{Bellman1953}, \cite{Wong1975} and \cite{Fowler1930}. To the best of our knowledge, \cite{Mehta1971} is the only work that discusses asymptotic behavior in the case $n=-1$, however not the same asymptotic behavior as we study in this paper. More precisely, the authors of \cite{Mehta1971} study the more general Emden-Fowler type equation with $h(u)=u^n,\ n\in\mathbb{R},\ \rho+\sigma = 0$ and minus sign in \eqref{eq:general_fowler_emden}.  In \cite{Mehta1971}, the more general equation appears in the context of the theory of diffusion and reaction governing the concentration $u$ of a substance disappearing by an 
isothermal reaction at each point $t$ of a slab of catalyst. When such an equation is normalized so that $u(t)$ is the concentration as a fraction of the concentration outside of the slab and $t$ the distance from the central plane as a fraction of the half thickness of the slab, the parameter $\sqrt{k}$ may be interpreted as the ratio of the characteristic reaction rate to the characteristic diffusion rate. This ratio is known in the chemical engineering literature as the Thiele modulus. In this context, it is natural to keep the range of $t$ finite and solve for the Thiele modulus as a function of the concentration of the substance $u$. Therefore, \cite{Mehta1971} studies the more general Emden-Fowler type equation for $u$ as a function of $\sqrt{k}$ and study asymptotic properties of the solution as $k\to\infty$. However, here we solve an Emden-Fowler equation for the special case $n=-1$ and for any given Thiele modulus $k$, and study what happens to the concentration $u(t)$ as $t$ goes to infinity, rather than $k$ to infinity.

Although the literature devoted to continuous Emden-Fowler equations and generalizations is very rich, there are not many papers related to the discrete Emden-Fowler equation \eqref{eq:voltages_distflow} or to more general second-order non-linear discrete equations of Emden-Fowler type within the following meaning. Let $j_0$ be a natural number and let $\mathbb{N}(j_0)$ denote the set of all natural numbers greater than or equal to a fixed integer $j_0$, that is,
\begin{align*}
\mathbb{N}(j_0):=\{j_0,j_0+1,\ldots\}.
\end{align*} Then, a second-order non-linear discrete equation of Emden-Fowler type
\begin{align}
\Delta^2 u(j)\pm j^{\alpha}u^m(j) = 0,\label{eq:general_discrete_emden_fowler}
\end{align} is studied, where $u:\mathbb{N}(j_0)\to\mathbb{R}$ is an unknown solution, $\Delta u(j):=u(j+1)-u(j)$ is its first-order forward difference, $\Delta^2 u(j):= \Delta(\Delta u(j))=u(j+2)-2u(j+1)+u(j)$ is its second-order forward difference, and $\alpha,m$ are real numbers. A function $u^*:\mathbb{N}(j_0)\to\mathbb{R}$ is called a solution of \eqref{eq:general_discrete_emden_fowler} if the equality
\begin{align*}
\Delta^2 u^*(j)\pm j^{\alpha}(u^*(j))^m = 0
\end{align*} holds for every $j\in\mathbb{N}(j_0)$.
The work done in this area focuses on finding conditions that guarantee the existence of a solution of such discrete equations. In \cite{Diblik2009}, the authors consider the special case of \eqref{eq:general_discrete_emden_fowler} where $\alpha = -2$, write it as a system of two difference equations, and prove a general theorem for this that gives sufficient conditions that guarantee the existence of at least one solution. In \cite{Akin-Bohnera2003, Erbe2012}, the authors replace the term $j^{\alpha}$ in \eqref{eq:general_discrete_emden_fowler} by $p(j)$, where the function $p(j)$ satisfies some technical conditions, and find conditions that guarantee the existence of a non-oscillatory solution. In \cite{Astashova2021,Migda2019}, the authors find conditions under which the nonlinear discrete equation in \eqref{eq:general_discrete_emden_fowler} with $m$ of the form $p/q$ where $p$ and $q$ are integers such that the difference $p-q$ is odd, has solutions with asymptotic behavior when $j\to\infty$ that is similar to a power-type function, that is,
\begin{align*}
u(j)\sim a_{\pm}j^{-s},\quad j\to\infty,
\end{align*} for constants $a_{\pm}$ and $s$ defined in terms of $\alpha$ and $m$. However, we study the case $m=-1$ and this does not meet the condition that $m$ is of the form $p/q$ where $p$ and $q$ are integers such that the difference $p-q$ is odd. 

The paper is structured as follows. 
In Section \ref{subsec:background_voltages}, we present the application that motivated our study of particular equations in \eqref{eq:voltages_approx} and \eqref{eq:voltages_distflow}.
We present the main results in two separate sections. In Section \ref{SEC:ASYMP_F(T)}, we present the asymptotic behavior and associated properties of the continuous solution of the differential equation in \eqref{eq:voltages_approx}, while in Section \ref{SEC:DISCRETE_RESULTS}, we present the asymptotic behavior of the discrete recursion in \eqref{eq:voltages_distflow}. The proofs of the main results in the continuous case, except for the results of Section \ref{SUBSEC:ASSOCIATED_PROPERTIES}, and discrete case can be found in Sections \ref{SEC:PROOFS_CONTINUOUS} and \ref{sec:proofs_discrete}, respectively. We finish the paper with a conclusion in Section \ref{sec:conclusion}. In the appendices, we gather the proofs for the results in Section \ref{SUBSEC:ASSOCIATED_PROPERTIES}. 

\section{Background on motivational application}\label{subsec:background_voltages}
Equation \eqref{eq:voltages_approx} emerges in the process of charging electric vehicles (EVs) by considering their random arrivals, their stochastic demand for energy at charging stations, and the characteristics of the electricity \emph{distribution network}. This process can be modeled as a queue, with EVs representing \emph{jobs}, and charging stations classified as \emph{servers}, constrained by the physical limitations of the distribution network \cite{Aveklouris2019b,Christianen2021}. 

An electric grid is a connected network that transfers electricity from producers to consumers. It consists of generating stations that produce electric power, high voltage transmission lines that carry power from distant sources to demand centers, and distribution lines that connect individual customers, e.g., houses, charging stations, etc. We focus on a network that connects a generator to charging stations with only distribution lines. Such a network is called a distribution network.

In a distribution network, distribution lines have an impedance, which results to voltage loss during transportation. Controlling the voltage loss ensures that every customer receives safe and reliable energy \cite{Kerstinga}. Therefore, an important constraint in a distribution network is the requirement of keeping voltage drops on a line under control.

In our setting, we assume that the distribution network, consisting of one generator, several charging stations and distribution lines with the same physical properties, has a line topology. The generator that produces electricity is called the \emph{root node}. Charging stations consume power and are called the \emph{load nodes}. Thus, we represent the distribution network by a graph (here, a line) with a root node, load nodes, and edges representing the distribution lines. Furthermore, we assume that EVs arrive at the same rate at each charging station. 

In order to model the power flow in the network, we use an approximation of the alternating current (AC) power flow equations \cite{Molzahn2019}. These power flow equations characterize the steady-state relationship between power injections at each node, the voltage magnitudes, and phase angles that are necessary to transmit power from generators to load nodes. We study a load flow model known as the \emph{branch flow model} or the \emph{Distflow model} \cite{Low2014d,BaranWu1989}. Due to the specific choice for the network as a line, the same arrival rate at all charging stations, distribution lines with the same physical properties, and the voltage drop constraint, the power flow model has a recursive structure, that is, the voltages at nodes $j=0,\ldots,N-1$, are given by recursion \eqref{eq:voltages_distflow}. Here, $N$ is the root node, and $V_0=1$ is chosen as normalization.
 This recursion leads to real-valued voltages and ignores line reactances and reactive power, which is a reasonable assumption in distribution networks. We refer to \cite{Christianen2021} for more detail. 

\section{Main results of continuous Emden-Fowler type equation}\label{SEC:ASYMP_F(T)}

In this section, we study the asymptotic behavior of the solution $f$ of \eqref{eq:voltages_approx}. To do so, we present in Lemma \ref{lemma:solution_f} the solution of a more general differential equation. Namely, we consider a more general initial condition $f(0)=y>0$.

The solution $f$ presented in Lemma \ref{lemma:solution_f} allows us to study the asymptotic behavior of $f_0(x)$, i.e., the solution of the differential equation in Lemma \ref{lemma:solution_f} where $k=1, y=1$ and $w=0$, or in other words, the solution of the differential equation $f''(x)=1/f(x)$ with initial conditions $f(0)=1$ and $f'(0)=0$; see Theorem \ref{THM:LIMITING_BEHAVIOR}. We can then derive the asymptotic behavior of $f$; see Corollary \ref{corollary:asymp_f}. 

The following theorem provides the limiting behavior of $f_0(x)$, i.e., the solution of Equation \eqref{eq:voltages_approx} where $k=1, y=1$ and $w=0$.
\begin{theorem}
Let $f_0(x)$ be the solution of \eqref{eq:voltages_approx} for $k=1, y=1$ and $w=0$.  The limiting behavior of the function $f_0(x)$ as $x\to\infty$ is given by,
\begin{align*}
f_0(x) = z(\ln(z))^{\frac{1}{2}}\left[1+\mathcal{O}\left(\frac{\ln(\ln(z))}{\ln(z)} \right) \right]
\end{align*} where $z=x\sqrt{2}$.
\label{THM:LIMITING_BEHAVIOR}
\end{theorem} We first derive an implicit solution to Equation \eqref{eq:voltages_approx} where $k=1, y=1$ and $w=0$. Namely, we derive $f_0(x)$ in terms of a function $U(x)$; cf.\ Lemma \ref{lemma:solution_f}. We show, using Lemma \ref{lemma:ineq_I(y)}, that we can derive an approximation of $U(x)$ by iterating the following equation: 
\begin{align}
\frac{\exp(U^2)-1}{2U} = \frac{x}{\sqrt{2}}.\label{eq:bound_iterative_method_behavior}
\end{align} We can then use this approximation of $U(x)$ in the implicit solution of the differential equation to derive the asymptotic behavior of Theorem \ref{THM:LIMITING_BEHAVIOR}. The proofs of Theorem \ref{THM:LIMITING_BEHAVIOR} and Lemma \ref{lemma:ineq_I(y)} can be found in Section \ref{SEC:PROOFS_CONTINUOUS}. We now give the necessary lemmas for the proof of Theorem \ref{THM:LIMITING_BEHAVIOR}.


\begin{lemma}[Lemma D.1 in \cite{Christianen2021}]\label{lemma:solution_f}
For $t\geq 0,k>0,y>0,w\geq 0$, the nonlinear differential equation
\begin{align*}
f''(t) = \frac{k}{f(t)}
\end{align*} with initial conditions $f(0)=y$ and $f'(0)=w$ has the unique solution 
\begin{align} 
f(t) = cf_0(a+bt).\label{eq:f}
\end{align} Here, $f_0$ is given by
\begin{align}\label{eq:f_0(x)}
f_0(x) = \exp(U^2(x)),\quad \text{for}~x\geq 0,
\end{align} where $U(x)$, for $x\geq 0$, is given by
\begin{align}\label{eq:Ux}
\int_0^{U(x)}\exp(u^2)~du = \frac{x}{\sqrt{2}},
\end{align}and where the constants $a,b,c$ are given by
\begin{align}
a & = \sqrt{2}\int_0^\frac{w}{\sqrt{2k}} \exp(u^2)~du, \label{eq:a}\\
b & = \frac{\sqrt{k}}{y}\exp\left(\frac{w^2}{2k}\right),\label{eq:b}\\
c & = y\exp\left(\frac{-w^2}{2k} \right).\label{eq:c}
\end{align}
\label{LEMMA:DIFF_EQ1}
\end{lemma}

Notice that we do not find an elementary closed-form solution of the function $f_0(x)$, since $f_0(x)$ is given in terms of $U(x)$, given implicitly by \eqref{eq:Ux}. For $x\geq 0$, the left-hand side of \eqref{eq:Ux} is equal to $\frac{1}{2}\sqrt{\pi} \text{erfi}(U(x))$ where  $\text{erfi}(z)$ is the imaginary error function, defined by 
\begin{align}
\text{erfi}(z) = -\mathrm{i}\ \text{erf}(\mathrm{i}z),
\end{align} where $\text{erf}(w) = \frac{2}{\sqrt{\pi}}\int_0^w \exp(-v^2)dv$ is the well-known error function. 

\begin{lemma}\label{lemma:ineq_I(y)}
For $y\geq 0$, we have the inequalities
\begin{align}
\frac{\exp(y^2)-1}{2y}\leq \int_0^y \exp(u^2)du\leq \frac{\exp(y^2)-1}{y},\label{eq:inequalities_int_exp}
\end{align} and
\begin{align}
\int_0^y \exp(u^2)du \leq \frac{\exp(y^2)-1}{2y}\left(1+\frac{2}{y^2} \right).\label{eq:inequality_exp}
\end{align}
\end{lemma}

Now, we present the asymptotic behavior of the solution $f$ of \eqref{eq:voltages_approx}.

\begin{corollary}\label{corollary:asymp_f}
The limiting behavior of the function $f(t)$, defined in Equation \eqref{eq:f}, is given by
\begin{align}
f(t)=t\sqrt{2k\ln(t)}\left(1+\mathcal{O}\left(\frac{\ln(\ln(t))}{\ln(t)} \right)\right),\quad t\to\infty.\label{eq:f(t)_big_O}
\end{align}
\end{corollary}

\begin{proof}[Proof of Corollary \ref{corollary:asymp_f}]
In order to derive a limit result of the exact solution of \eqref{eq:voltages_approx}, i.e. for \eqref{eq:f} with initial conditions $f(0)=1$ and $f'(0)=w$, we use the limiting behavior of the function $f_0(x)$ and the definitions of $a,b$ and $c$ as in \eqref{eq:a}--\eqref{eq:c}. Denote $v = \ln(z)$. Then, by Theorem \ref{THM:LIMITING_BEHAVIOR}, we have
\begin{align}
f(t) = cf_0(a+bt) = czv^{\frac{1}{2}}\left(1+\mathcal{O}\left(\frac{\ln(v)}{v} \right) \right).\label{eq:put_together_ft}
\end{align} In what follows, we carefully examine the quantities $czv^{\frac{1}{2}}$ and $\ln(v)/v$. First, observe that
\begin{align*}
v = \ln(z) = \ln((a+bt)\sqrt{2}) = \ln(t)+\mathcal{O}(1),\quad t>\exp(1),
\end{align*} which yields
\begin{align*}
v^{\frac{1}{2}} & = \left(\ln(t)+\mathcal{O}(1)\right)^{\frac{1}{2}} \\
& = \ln(t)^{\frac{1}{2}}\left(1+\mathcal{O}\left(\frac{1}{\ln(t)}\right) \right),\quad t>\exp(1),
\end{align*} and
\begin{align*}
\ln(v) & = \ln(\ln(t)+\mathcal{O}(1)) \\
& = \ln(\ln(t))+\mathcal{O}\left(\frac{1}{\ln(t)}\right),\quad t>\exp(1).
\end{align*} Therefore, using that $cb=\sqrt{k}$, we get
\begin{align}
czv^{\frac{1}{2}} & = c(a+bt)\sqrt{2}\ln(t)^{\frac{1}{2}}\left(1+\mathcal{O}\left(\frac{1}{\ln(t)}\right) \right) \nonumber\\
& = (t+\mathcal{O}(1))\sqrt{2k\ln(t)}\left(1+\mathcal{O}\left(\frac{1}{\ln(t)}\right) \right) \nonumber \\
& = t\sqrt{2k\ln(t)}\left(1+\mathcal{O}\left(\frac{1}{\ln(t)}\right) \right),\quad t>\exp(1),\label{eq:czsqrt(v)}
\end{align} and
\begin{align}
\frac{\ln(v)}{v} & = \frac{\ln(\ln(t))+\mathcal{O}\left(\frac{1}{\ln(t)} \right)}{\ln(t)+\mathcal{O}(1)} \nonumber \\
& = \frac{\ln(\ln(t))}{\ln(t)}\left(1+\mathcal{O}\left(\frac{1}{\ln(\ln(t))} \right) \right),\quad t>\exp(1).\label{eq:lnv_v}
\end{align} Putting the results in \eqref{eq:czsqrt(v)} and \eqref{eq:lnv_v} together in \eqref{eq:put_together_ft}, yields
\begin{align*}
f(t) = t\sqrt{2k\ln(t)}\left(1+\mathcal{O}\left(\frac{\ln(\ln(t))}{\ln(t)}\right) \right),\quad t>\exp(1).
\end{align*}
\end{proof}

\subsection{Associated properties of the ratio between $f$ and its first order approximation}\label{SUBSEC:ASSOCIATED_PROPERTIES}
In this section, we study associated properties of the ratio between $f(t)$ and its first order approximation. Using only the first term of the asymptotic expansion of \eqref{eq:f(t)_big_O}, we define
\begin{align}
g(t):= t\sqrt{2k\ln(t)}.\label{eq:f(t)_approx}
\end{align}
 
The reason for studying this ratio, and in particular the role of $k$, is twofold: (1) the useful insights that we get for (the proof of) the asymptotic behavior in the discrete case in Section \ref{SEC:DISCRETE_RESULTS}, and (2) the applicability of Equation \eqref{eq:voltages_approx} in our motivational application, in cases where the parameter $k$ in \eqref{eq:voltages_approx} is small. 


Considering the practical application for charging electric vehicles, the ratio of normalized voltages $V_j/V_0 = V_j, j=1,2,\ldots$ should be below a level $1/(1-\Delta)$, where the tolerance $\Delta$ is small (of the order $10^{-1}$), due to the voltage drop constraint. Therefore, the parameter $k$, comprising given charging rates and resistances at all stations, is normally small (of the order $10^{-3}$).

Furthermore, to match the initial conditions $V_0=1$ and $V_1 = 1+k$ of the discrete recursion with the initial conditions of the continuous analog, we demand $f(0)=1$ and $f(1) = 1+k$.
However, notice that in our continuous analog described by \eqref{eq:voltages_approx}, we have, next to the initial condition $f(0)=1$, the initial condition $f'(0)=w$, while nothing is assumed about the value $f(1)$. The question arises whether it is possible to connect the conditions $f'(0)=w$ and $f(1)=1+k$. To do so, we use an alternative representation of $f$ given in Lemma \ref{lemma:alternative_f}. Then, using this representation, we show the existence and uniqueness of $w\geq 0$ for every $k$ such that the solution of \eqref{eq:voltages_approx} satisfies $f(1)=1+k$ in Lemma \ref{lemma:existence_uniqueness_w}. The proof of Lemmas \ref{lemma:alternative_f}--\ref{lemma:existence_uniqueness_w} can be found in Appendix \ref{sec:existence_uniqueness_w}.

The importance of the role of the parameter $k$ becomes immediate from the comparison of the functions $f(t)$ and $g(t)$ in Theorem \ref{thm:cases_k}.

\begin{theorem}\label{thm:cases_k}
Let $f(t)$ be given by \eqref{eq:f} with initial conditions $f(0)=1$, $f'(0)=w$ such that $f(1)=1+k$, and let $g(t)$ be given by \eqref{eq:f(t)_approx}. Then, there is a unique $k_c = 1.0384\ldots$ such that
\begin{enumerate}[label=(\alph*)]
\item $k\geq k_c$ implies $f(t)\geq g(t)$ for all $t\geq 1$,
\item $0<k<k_c$ implies that there are $t_1(k),t_2(k)$ with $1<t_1(k)<t_2(k)<\infty$ such that $f(t)<g(t)$ when $t_1(k)<t<t_2(k)$ and $f(t)>g(t)$ when $1\leq t<t_1(k)$ or $t>t_2(k)$. 
\end{enumerate}
\end{theorem}

In what follows, we start with introducing notation for the proof of Theorem \ref{thm:cases_k}, and give a sketch of the proof. The theorem is proven in Appendix \ref{sec:existence_uniqueness_w}.

Define the auxiliary function $\psi:[1,\infty)\to [0,\infty)$ by
\begin{align}
\psi(t):=2k+\frac{k}{2\ln(t)}-k\ln(2k\ln(t)),\label{eq:psi}
\end{align} and notice (also for the proof in Lemma \ref{lemma:equivalence}) that the function $\psi(t)$ is strictly decreasing from $+\infty$ at $t=1$ to $0$ at $t=\infty$. This follows easily from the definition of $\psi$ in \eqref{eq:psi}. 

Denote the unique solution $t>1$ of the equation $ \psi(t) = w^2$ by $t_0(k)$, i.e.
\begin{align}
\psi(t_0(k)) = w^2,\label{eq:equation_psi}
\end{align} where $w$ comes from the initial condition $f'(0)=w\geq 0$. Additionally, define
\begin{align}
F(t,k) &:= \int_{(W^2+\ln(f(t)))^{\frac{1}{2}}}^{(W^2+\ln(g(t)))^{\frac{1}{2}}} \exp(v^2)dv \label{eq:def_F}\\
& = -t\sqrt{\frac{k}{2}}\exp(W^2)+\int_W^{(W^2+\ln(g(t)))^{\frac{1}{2}}}\exp(v^2)dv,\label{eq:F(t,k)}
\end{align} where the second line is a consequence of Lemma \ref{lemma:alternative_f} with $y=1$. The proof of Theorem \ref{thm:cases_k} centers about the unique solution $t_0(k)$ of \eqref{eq:equation_psi}. First, from \eqref{eq:def_F}, we notice that $\max_{t\geq 1} F(t,k)\leq 0$ is equivalent to $f(t)\geq g(t)$. In Lemma \ref{lemma:F(t,k)}, we show that $\max_{t\geq 1}F(t,k)$ is exactly attained at the point $t_0(k)$, i.e., $\max_{t\geq 1}F(t,k) = F(t_0(k),k)$. Notice that $F(t_0(k),k)$ is only a function of the parameter $k$. In Lemma \ref{lemma:F(t,k)_decreasing}, we show $F(t_0(k),k)$ is a strictly decreasing function of $k$. To prove Lemma \ref{lemma:F(t,k)_decreasing}, we make use of additional Lemma \ref{lemma:increasing_W}. Then, in Lemmas \ref{lemma:positive_small_k} and \ref{lemma:negative_large_k}, we show that $F(t_0(k),k)$ is positive for small $k$ and negative for large $k$, respectively. This allows us to conclude that $F(t_0(k),k)\leq 0$ is equivalent to $k\geq k_c$. In summary, to prove Theorem  \ref{thm:cases_k}, we show
\begin{align*}
f(t)\geq g(t) & \iff \max_{t\geq 1} F(t,k) = F(t_0(k),k) \leq 0 \iff k\geq k_c.
\end{align*} Furthermore, in Lemma \ref{lemma:F(t,k)}, we show that $F(t,k)$ has only one extreme point, and in particular that this extreme point is a maximum and that this is attained at the point $t_0(k)$. Thus, in the case where $0<k<k_c$, we are left with $t_1(k),t_2(k)$ with $1<t_1(k)<t_2(k)<\infty$ such that $f(t)<g(t)$ when $t_1(k)<t<t_2(k)$ and $f(t)>g(t)$ when $1\leq t<t_1(k)$ or $t>t_2(k)$.

Necessary Lemmas \ref{lemma:F(t,k)}--\ref{lemma:negative_large_k} to prove Theorem \ref{thm:cases_k} are stated and proven in Appendix \ref{sec:existence_uniqueness_w}.

A comparison of the approximation $g(t)$, i.e. for \eqref{eq:f(t)_approx}, to the exact solution $f(t)$ of \eqref{eq:voltages_approx} where $w$ is such that $f(1)=1+k$, for three values of $k$, is given in Figure \ref{fig:quotient_f_g_500_log_a}.

\begin{figure}[h]
  \centering
  \includegraphics[scale=0.5]{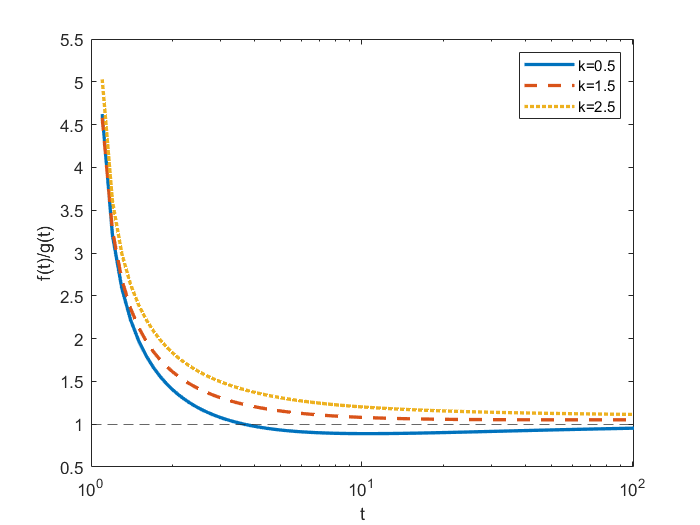}
  \caption{Plot of quotient $f/g$ for three values of $k$.}
  \label{fig:quotient_f_g_500_log_a}
\end{figure}

However, in the setting where $k$ is small, the result in Theorem \ref{thm:cases_k}, case (b) leaves two practical questions; how small the ratio $f(t)/g(t)$ can be when $t_1(k)\leq t\leq t_2(k)$ and how large the ratio $f(t)/g(t)$ can be when $t\geq t_2(k)$. These practical questions are covered in Theorem \ref{thm:bounds_f/g}.


\begin{theorem}\label{thm:bounds_f/g}
Let $f(t)$ be given by \eqref{eq:f} with initial conditions $f(0)=1, f'(0)=w$ such that $f(1)=1+k$, and let $g(t)$ be given by \eqref{eq:f(t)_approx}. Then, for $0<k<k_c$, we have
\begin{align}
f(t)/g(t) \geq \frac{1}{2}\left(\ln\left(\sqrt{2/k} \right) \right)^{-\frac{1}{2}},\label{eq:lowerbound_f_g}
\end{align} when $t_1(k)\leq t\leq t_2(k)$. Furthermore, 
we have 
\begin{align}
f(t)/g(t)\leq 1.21 \label{eq:upperbound_f_g}
\end{align} when $t\geq t_2(k)$.
\end{theorem}

The proof exploits properties of Theorem \ref{lemma:solution_f} and Theorem \ref{thm:cases_k}, such as exact representations \eqref{eq:a}--\eqref{eq:c} and actual values such as the one for $k_c$, but most importantly, we use numerical results to compute bounds for the quantity $f_0(x)/g(x)$, where $f_0(x)$ is given in \eqref{eq:f_0(x)} and $g(x)$ is given in \eqref{eq:f(t)_approx}. The proofs of Theorem \ref{thm:bounds_f/g}, and supporting Lemmas \ref{lemma:maximum_ratio_f0_g} and \ref{lemma:t2(k)} can be found in Appendix \ref{sec:existence_uniqueness_w}.

\section{Main results of discrete Emden-Fowler type equation}\label{SEC:DISCRETE_RESULTS}

In this section, we present the asymptotic behavior of the discrete recursion \eqref{eq:voltages_distflow}. Thus, we consider the sequence $V_j, j=0,1,\ldots$ defined in \eqref{eq:voltages_distflow} 
and we let
\begin{align}
W_j = j\left(2k\ln(j) \right)^{\frac{1}{2}}=g(j),\quad j=1,2,\ldots,\label{eq:Wn_sequence}
\end{align} denote the discrete equation analog to $g(j)$; cf.\ \eqref{eq:f(t)_approx}, at integer points $j=1,2,\ldots$. The asymptotic behavior of the discrete recursion \eqref{eq:voltages_distflow} is summarized in the following theorem.
\begin{theorem}\label{thm:lim_V_j_W_j}
Let $V_j,j=0,1,\ldots$ and $W_j,j=1,2,\ldots$ be as in \eqref{eq:voltages_distflow} and \eqref{eq:Wn_sequence}, respectively. Then,
\begin{align*}
\lim_{j\to\infty} \frac{V_j}{W_j} = 1.
\end{align*}
\end{theorem}
 

The proof of Theorem \ref{thm:lim_V_j_W_j} relies on the following observations: there always exists a point $n\in\{1,2,\ldots\}$ such that either $V_j\geq W_j$ for all $j\geq n$ or $V_j\leq W_j$ for all $j\geq n$, and the existence of such a point implies in either case the desired asymptotic behavior of the sequence $V_j$.

To show that there exists a point $n\in\{1,2,\ldots\}$ such that either $V_j\geq W_j$ for all $j\geq n$ or $V_j\leq W_j$ for all $j\geq n$, we rely on Lemmas \ref{lemma:upper_bound_W}, \ref{lemma:lower_bound_V} and \ref{lemma:equivalence}. Due to the inequalities in Lemmas \ref{lemma:upper_bound_W} and \ref{lemma:lower_bound_V}, we show
\begin{align}
V_{j+1}-V_j \geq W_{j+1}-W_j,\label{eq:ineq_V_W}
\end{align} for $j\geq n_0(k)$, where $n_0(k)$ is appropriately chosen. Then, Equation \eqref{eq:ineq_V_W} implies that there exists either a point $n\geq n_0(k)$ such that $V_n\geq W_n$ or not. If there exists a point $n\geq n_0(k)$ such that $V_n\geq W_n$, then we show in Lemma \ref{lemma:equivalence} that $V_j\geq W_j$ for all $j\geq n$. If not, we have that $V_j<W_j$ for all $j\geq n_0(k)$. 

Then, we are left to show that the existence of such a point implies the desired asymptotic behavior of $V_j$. This is done in Lemma \ref{lemma:V_j_geq_W_j}.







We now give the necessary lemmas to prove Theorem \ref{thm:lim_V_j_W_j}.

\begin{lemma}\label{lemma:upper_bound_W} Let $W_j, j=1,2,\ldots$ be as in \eqref{eq:Wn_sequence}. Then,
\begin{align}
W_{j+1}-W_j \leq (\psi(j+1)+2k\ln(W_{j+1}))^{\frac{1}{2}}.\label{eq:difference_W_j}
\end{align} where $\psi(j)$ for $j=1,\ldots$ is defined in \eqref{eq:psi}.
\end{lemma}
\begin{lemma}\label{lemma:lower_bound_V} Let $V_j,j=0,1,\ldots$ be as in \eqref{eq:voltages_distflow}. Then,
\begin{align}
V_{j+1}-V_j & \geq \left(C+2k\ln(V_j) \right)^{\frac{1}{2}},\label{eq:difference_V_j}
\end{align} for some constant $C$.
\end{lemma} 

\begin{lemma}\label{lemma:equivalence}
Let $V_j,j=0,1,\ldots$ and $W_j,j=1,2,\ldots$ be as in \eqref{eq:voltages_distflow} and \eqref{eq:Wn_sequence}, respectively.  Then, we have the following equivalence.
\begin{enumerate}
\item There is $n\geq n_0(k)$ such that $V_n\geq W_n$. 
\item There is $n\geq n_0(k)$ such that $V_j\geq W_j$ for all $j\geq n$.
\end{enumerate}
\end{lemma}

\begin{lemma}\label{lemma:V_j_geq_W_j}
Let $V_j, j=0,1,\ldots,N$ and $W_j, j=1,2,\ldots,N$ be as in \eqref{eq:voltages_distflow} and \eqref{eq:Wn_sequence}, respectively. There holds the following. In either case that
\begin{enumerate}
\item there is a point $n\in\{1,2,\ldots\}$ such that $V_j\geq W_j$ for all $j\geq n$,
\end{enumerate} or
\begin{enumerate}[resume]
\item there is a point $n\in\{1,2,\ldots\}$ such that $V_j\leq W_j$ for all $j\geq n$,
\end{enumerate} we have that $V_j= W_j(1+\smallO(1)), j\to\infty$.
\end{lemma}
The proofs of Lemmas \ref{lemma:upper_bound_W}--\ref{lemma:V_j_geq_W_j} are given in Section \ref{sec:proofs_discrete}. Now, Theorem \ref{thm:lim_V_j_W_j} follows from Lemmas \ref{lemma:upper_bound_W}--\ref{lemma:V_j_geq_W_j}.
\begin{proof}[Proof of Theorem \ref{thm:lim_V_j_W_j}]
Let $V_j,j=0,1,\ldots$ and $W_j,j=1,2,\ldots$ be as in \eqref{eq:voltages_distflow} and \eqref{eq:Wn_sequence}, respectively. 
On the one hand, as a result of Lemma \ref{lemma:lower_bound_V}, the first order differences of the sequence $V_j$ are bounded according to \eqref{eq:difference_V_j}, while on the other hand, as a result of Lemma \ref{lemma:upper_bound_W}, the first order finite differences of $W_j$ are bounded according to \eqref{eq:difference_W_j}. 

A minor issue is that \eqref{eq:difference_W_j} involves $\ln(W_{j+1})$, whereas \eqref{eq:difference_V_j} involves $\ln(V_j)$. However, by \eqref{eq:Wn_sequence}, we write
\begin{align}
\ln(W_{j+1})-\ln(W_j) & = \ln\left((j+1)(2k\ln(j+1))^{\frac{1}{2}}\right)-\ln\left(j(2k\ln(j))^{\frac{1}{2}}\right) \nonumber\\
& = \ln\left(1+\frac{1}{j}\right)+\frac{1}{2}\ln\left(\left(\frac{\ln(j+1)}{\ln(j)}\right) \right)\label{eq:difference_logs_W}
\end{align} and notice from increasingness of the function $j\geq 1 \mapsto \ln(j)$ and the inequality $\ln(j+1)-\ln(j)\leq \frac{1}{j}$ that $\frac{\ln(j+1)}{\ln(j)}\leq 1+\frac{1}{j}$ when $j>\exp(1)$. Using this last inequality in \eqref{eq:difference_logs_W}, yields that $\ln(W_{j+1}) = \ln(W_j)+\mathcal{O}(1/j)$. 

Moreover, Equations \eqref{eq:difference_W_j} and \eqref{eq:difference_V_j} imply that there exists a point $n_0(k)$ such that $V_{j+1}-V_j\geq W_{j+1}-W_j$ when $j\geq n_0(k)$. To eliminate the effect of the term $\mathcal{O}(1/j)$ in $\ln(W_{j+1})=\ln(W_j)+\mathcal{O}(1/j)$, we let $n_0(k)$ be such that $\psi(n_0(k))\leq C-1$. 

In any case, we can distinguish between two cases: there exists either a point $n\geq n_0(k)$ such that $V_n\geq W_n$ or not, i.e., 
\begin{enumerate}
\item There is $n\geq n_0(k)$ such that $V_{n}\geq W_{n}$,
\item $V_{j}< W_j$ for all $j\geq n_0(k)$.
\end{enumerate}
By Lemma \ref{lemma:equivalence}, we have, on the one hand, that the existence of a point $n\geq n_0(k)$ such that $V_n\geq W_n$, implies that $V_j\geq W_j$ for all $j\geq n$ and on the other hand, that the non-existence of $n\geq n_0(k)$ such that $V_n\geq W_n$, implies that $V_j< W_j$ for all $j\geq n_0(k)$.

This situation exactly fits the framework of Lemma \ref{lemma:V_j_geq_W_j}.
 

We consider the two cases above. First, assume that (1) holds. Then by Lemma \ref{lemma:equivalence}, we have $V_j\geq W_j$ for all $j\geq n$. From $V_j\geq W_j$, for all $j\geq n$, 
we have that Lemma \ref{lemma:V_j_geq_W_j}, item (1) holds, and so
\begin{align*}
V_j = W_j(1+\smallO(1)),\quad j\to\infty.
\end{align*} 
Second, assume that (2) holds, so that $V_j< W_j$ for all $j>n_0(k)$. Then, Lemma \ref{lemma:V_j_geq_W_j}, item (2) holds, and so
\begin{align*}
V_j= W_j(1+\smallO(1)),\quad j\to\infty.
\end{align*} Hence, any of the two cases yields
\begin{align*}
\lim_{j\to\infty} \frac{V_j}{W_j} = 1.
\end{align*}
\end{proof}

Although we do not provide associated properties of the asymptotic behavior of $V_j$ as $j\to\infty$ as we did for the asymptotic behavior of $f(x)$ as $x\to\infty$, we compare the behavior of $V_j$ with the discrete counterpart of $g(t)$, i.e. $W_j$, for $j=1,\ldots,100$ in Figure \ref{fig:quotient_V_W_500_log_a}.

\begin{figure}[h]
  \centering
  \includegraphics[scale=0.5]{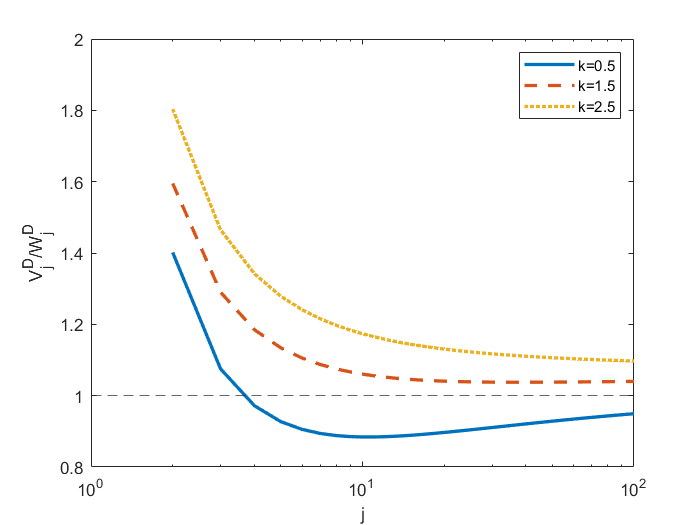}
  \caption{Plot of quotient $V/W$ for three values of $k$.}
  \label{fig:quotient_V_W_500_log_a}
\end{figure}

\section{Proofs for Section \ref{SEC:ASYMP_F(T)}}\label{SEC:PROOFS_CONTINUOUS}
The main result in Section \ref{SEC:ASYMP_F(T)}, i.e., Theorem \ref{THM:LIMITING_BEHAVIOR} follows from Lemmas \ref{lemma:solution_f} and \ref{lemma:ineq_I(y)}. In this section, we provide the proofs of both Theorem \ref{THM:LIMITING_BEHAVIOR} and Lemma \ref{lemma:ineq_I(y)}. For the proof of Lemma \ref{lemma:solution_f} we refer to \cite{Christianen2021}.

\subsection{Proof of Theorem \ref{THM:LIMITING_BEHAVIOR}}\label{subsec:thm_limiting_behavior}

\begin{proof}[Proof of Theorem \ref{THM:LIMITING_BEHAVIOR}]
Denoting $U=U(x)=\left(\ln(f_0(x)\right)^{\frac{1}{2}}$ for $x\geq 0$, we have by \eqref{eq:Ux}
\begin{align}
\int_0^U \exp(u^2)du = \frac{x}{\sqrt{2}}.\label{eq:intY}
\end{align} We consider for $x\geq 0$ the equation
\begin{align}
\frac{\exp(y^2)-1}{2y} = \frac{x}{\sqrt{2}}.\label{eq:B20}
\end{align} With $z=x\sqrt{2}$, we can write \eqref{eq:B20} as
\begin{align*}
y = h_z(y), h_z(y)=\left(\ln(1+zy) \right)^{\frac{1}{2}}.
\end{align*} The function $h_z(y)$ is concave in $y\geq 0$ since
\begin{align*}
\frac{d}{dy}\left[h_z(y) \right] = \frac{z}{2(1+yz)(\ln(1+yz))^{\frac{1}{2}}}
\end{align*} is decreasing in $y\geq 0$. Furthermore, when $z>\exp(1)-1$,
\begin{align*}
h_z(1) = \left(\ln(1+z) \right)^{\frac{1}{2}}>1, h_z(z)=(\ln(1+z^2))^{\frac{1}{2}}<z,
\end{align*} where the first inequality follows from $z>\exp(1)-1$ and the second inequality follows from $\ln(1+z^2)<z^2,z>0$. Therefore, the equation $y=h_z(y)$ has for any $z>\exp(1)-1$ exactly one solution $y_{LB}\in[1,z]$; here ``LB" refers to the lower-bound in \eqref{eq:inequalities_int_exp}. Since $y_{LB}\in[1,z]$, we have
\begin{align}
y_{LB}=(\ln(1+zy_{LB}))^{\frac{1}{2}}\in \left[(\ln(1+z))^{\frac{1}{2}},(\ln(1+z^2))^{\frac{1}{2}} \right],\label{eq:B24}
\end{align} so that $y_{LB}=\mathcal{O}(\ln(z)^{\frac{1}{2}}),z>\exp(1)-1$. When we iterate \eqref{eq:B24} one more time, we get
\begin{align}
y_{LB} & =\left(\ln(z)+\ln\left(\frac{1}{z}+y_{LB}\right) \right)^{\frac{1}{2}}=(\ln(z))^{\frac{1}{2}}\left(1+\frac{\ln\left(\frac{1}{z}+y_{LB}\right)}{\ln(z)} \right)^{\frac{1}{2}} \nonumber \\
& = (\ln(z))^{\frac{1}{2}}\left(1+\mathcal{O}\left(\frac{\ln(\ln(z))}{\ln(z)} \right) \right), \quad z>\exp(1)-1.\label{eq:B25}
\end{align} Observe that
\begin{align}
U = (\ln(f_0(x))^{\frac{1}{2}}\leq y_{LB},\quad z>\exp(1)-1.\label{eq:B26}
\end{align} Indeed, we have from \eqref{eq:intY} and the first inequality in \eqref{eq:inequalities_int_exp}
\begin{align}
\frac{\exp(U^2)-1}{2U}\leq \int_0^U \exp(u^2)du = \frac{x}{\sqrt{2}} = \frac{\exp(y_{LB}^2)-1}{2y_{LB}},\label{eq:B27}
\end{align} and so $U\leq y_{LB}$ follows from increasingness of the function $y\geq 0 \mapsto (\exp(y^2)-1)/2y$. In addition to the upper bound on $y$ in \eqref{eq:B26}, we also have the lower bound 
\begin{align}
U\geq \left(\ln\left(\frac{z}{2}\right) \right)^{\frac{1}{2}},\quad z\geq 2.\label{eq:B28}
\end{align} Indeed, from \eqref{eq:intY} and the second inequality in \eqref{eq:inequalities_int_exp},
\begin{align*}
\frac{\exp(U^2)-1}{U}\geq \int_0^U \exp(u^2)du=\frac{x}{\sqrt{2}}=\frac{z}{2},
\end{align*} while
\begin{align}
\frac{\exp(y^2)-1}{y}\bigg|_{y=(\ln(\frac{z}{2}))^{\frac{1}{2}}} = \frac{\frac{z}{2}-1}{(\ln(\frac{z}{2}))^{\frac{1}{2}}}\leq \frac{z}{2},\quad z\geq 2,\label{eq:B30}
\end{align} where the inequality in \eqref{eq:B30} follows from $-\ln(w)\geq (1-w)^2$ with $w=\frac{2}{z}\in (0,1]$. We have from \eqref{eq:B27} that 
\begin{align}
\frac{\exp(U^2)-1}{2U}\leq \frac{x}{\sqrt{2}}.\label{eq:B31}
\end{align}  When we use \eqref{eq:B28} in \eqref{eq:inequality_exp} with $y=U$, we see that
\begin{align}
\frac{x}{\sqrt{2}}=\int_0^U \exp(u^2)du & \leq \frac{\exp(U^2)-1}{2U}\left(1+\frac{2}{U^2} \right) \nonumber \\
& = \frac{\exp(U^2)-1}{2U}\left(1+\mathcal{O}\left(\frac{1}{\ln(z)}\right) \right).\label{eq:B32}
\end{align} From \eqref{eq:B31} and \eqref{eq:B32}, we then find that
\begin{align}
\frac{\exp(U^2)-1}{2U} = \frac{x}{\sqrt{2}}\left(1+\mathcal{O}\left(\frac{1}{\ln(z)}\right) \right).\label{eq:B33}
\end{align} Observe that \eqref{eq:B33} coincides with \eqref{eq:B20} when we take $y=U$ and replace the right-hand side $\frac{x}{\sqrt{2}}$ by $\left(\frac{x}{\sqrt{2}}\right)\left(1+\mathcal{O}\left(\frac{1}{\ln(z)} \right) \right)$. Using then \eqref{eq:B25} with $z$ replaced by $z\left(1+\mathcal{O}\left(\frac{1}{\ln(z)} \right) \right)$, we find that
\begin{align}
U & = \left(\ln\left(z\left(1+\mathcal{O}\left(\frac{1}{\ln(z)} \right) \right) \right)\right)^{\frac{1}{2}}\left(1+\mathcal{O}\left(\frac{\ln(\ln(z\left(1+\mathcal{O}\left(\frac{1}{\ln(z)} \right) \right)))}{\ln(z\left(1+\mathcal{O}\left(\frac{1}{\ln(z)} \right) \right))} \right) \right)\nonumber\\
& = (\ln(z))^{\frac{1}{2}}\left(1+\mathcal{O}\left(\frac{\ln(\ln(z))}{\ln(z)} \right) \right).\label{eq:B34}
\end{align} Then, finally, from \eqref{eq:B33} and \eqref{eq:B34},
\begin{align}
f_0(x) & = \exp(U^2) = 1+zU\left(1+\mathcal{O}\left(\frac{1}{\ln(z)} \right) \right)\nonumber\\
& = 1+z(\ln(z))^{\frac{1}{2}}\left(1+\mathcal{O}\left(\frac{\ln(\ln(z))}{\ln(z)} \right) \right)\left(1+\mathcal{O}\left(\frac{1}{\ln(z)} \right) \right)\nonumber \\
& = z(\ln(z))^{\frac{1}{2}}\left(1+\mathcal{O}\left(\frac{\ln(\ln(z))}{\ln(z)} \right) \right) \nonumber,
\end{align} as required.
\end{proof}

\subsection{Proof of Lemma \ref{lemma:ineq_I(y)}}\label{subsec:proof_lemma_ineq_I(y)}

\begin{proof}[Proof of Lemma \ref{lemma:ineq_I(y)}]
We require the inequalities \eqref{eq:inequalities_int_exp} and \eqref{eq:inequality_exp}. The inequalities in \eqref{eq:inequalities_int_exp} follow from expanding the three functions in \eqref{eq:inequalities_int_exp} as a series involving odd powers $y^{2l+1}, l=0,1,\ldots,$ of $y$ and comparing coefficients, i.e.,
\begin{align*}
\frac{\exp(y^2)-1}{2y} & = \sum_{\ell=0}^{\infty} \frac{y^{2\ell+1}}{2(\ell+1)!} \\
& \leq \sum_{\ell=0}^{\infty} \frac{y^{2\ell+1}}{(2\ell+1)\ell!} = \int_0^y \exp(u^2)du \\
& \leq \sum_{\ell=0}^{\infty} \frac{y^{2\ell+1}}{(\ell+1)!} = \frac{\exp(y^2)-1}{y}.
\end{align*} As to the inequality in \eqref{eq:inequality_exp}, we use partial integration according to
\begin{align}
\int_0^y \exp(u^2)du & = \int_0^y \frac{1}{2u}d(\exp(u^2)-1)\nonumber \\
& = \frac{\exp(y^2)-1}{2y}+\int_0^y \frac{\exp(u^2)-1}{2u^2}du.\label{eq:B18}
\end{align} Now
\begin{align}
\int_0^y \frac{\exp(u^2)-1}{2u^2}du \leq \frac{\exp(y^2)-1-y^2}{y^3},\quad y\geq 0, \label{eq:B19}
\end{align} as follows from expanding the two functions in \eqref{eq:B19} as a series involving odd powers $y^{2l+1}, l=0,1,\ldots,$ of $y$ and comparing coefficients. Then \eqref{eq:inequality_exp} follows from \eqref{eq:B18}--\eqref{eq:B19} upon deleting the $y^2$ in the numerator at the right-hand side of \eqref{eq:B19}. 
\end{proof}

\section{Proofs for Section \ref{SEC:DISCRETE_RESULTS}}\label{sec:proofs_discrete}
The main result in Section \ref{SEC:DISCRETE_RESULTS} follows from Lemmas \ref{lemma:upper_bound_W}--\ref{lemma:V_j_geq_W_j}. The proof of each Lemma can be found in \ref{subsubsec:W_j}--\ref{subsec:proof_lemma_V_j_geq_W_j}, respectively.
\subsection{Proof of Lemma \ref{lemma:upper_bound_W}}\label{subsubsec:W_j}
\begin{proof}[Proof of Lemma \ref{lemma:upper_bound_W}] For Equation \eqref{eq:Wn_sequence}, by the mean-value theorem, there is a $\xi\in[j,j+1]$ such that
\begin{align}
W_{j+1}-W_j & = g(j+1)-g(j) = g'(\xi) \leq g'(j+1) .\label{eq:W_j+1-W_j}
\end{align} We have used here that $g(j)$ is convex in $j\geq \exp(1/2)$. We make the term $g'(j+1)$ explicit, by differentiating $g(t)$, see \eqref{eq:f(t)_approx}, with respect to $t$ and rewrite it in terms of the function $g(t)$ itself (so for integer points, in terms of $W_t$) and $\psi(t)$ as in \eqref{eq:psi}. 
Differentiation of $g(t)$ gives,  
\begin{align}
g'(t) & = (2k)^{\frac{1}{2}}\frac{d}{dt}(t(\ln(t))^{\frac{1}{2}})\nonumber \\
& = (2k)^{\frac{1}{2}}\left((\ln(t))^{\frac{1}{2}}+\frac{1}{2(\ln(t))^{\frac{1}{2}}} \right)\nonumber\\
& = (2k)^{\frac{1}{2}}\left(\frac{(2\ln(t)+1)^2}{4\ln(t)} \right)^{\frac{1}{2}}\nonumber\\
& = \left(2k\ln(t)+2k+\frac{k}{2\ln(t)} \right)^{\frac{1}{2}}.\label{eq:diff_g(t)_right_form}
\end{align} However, \eqref{eq:diff_g(t)_right_form} does not contain the function $g(t)$ yet. Therefore, we rewrite the first term of the right-hand side of \eqref{eq:diff_g(t)_right_form} as follows:
\begin{align}
2k\ln(t) & = 2k\ln(t(2k\ln(t))^{\frac{1}{2}})-k\ln(2k\ln(t)) \nonumber\\
& = 2k\ln(g(t))-k\ln(2k\ln(t)).\label{eq:diff_g(t)_step}
\end{align} Then, after inserting \eqref{eq:diff_g(t)_step} and the definition of $\psi(t)$ in \eqref{eq:psi}, we get
\begin{align}
g'(t) & = \left(2k\ln(g(t))-k\ln(2k\ln(t)+2k+\frac{k}{2\ln(t)} \right)^{\frac{1}{2}}\nonumber\\
& = \left(\psi(t)+2k\ln(g(t)) \right)^{\frac{1}{2}},\quad  t>1.\label{eq:diff_g(t)}
\end{align}
Then, combining the upper bound in \eqref{eq:W_j+1-W_j} and \eqref{eq:diff_g(t)}, yields the desired upper bound for the finite differences of $W_j$ in \eqref{eq:difference_W_j}. 
\end{proof}

\subsection{Proof of Lemma \ref{lemma:lower_bound_V}}\label{subsubsecc:V_j} In this section, we prove a lower bound for the first order finite differences of $V_j$ that is similar to the upper bound we obtained in \eqref{eq:difference_W_j}. This result follows from Lemmas \ref{lemma:properties_Vn} and \ref{lemma:connect_to_ln}.

In more detail, the proof of Lemma \ref{lemma:lower_bound_V} consists of algebraic manipulations of \eqref{eq:voltages_distflow}, but the key in the proof is the use of Lemma \ref{lemma:connect_to_ln} in these manipulations, which, in turn, builds on technical results established in Lemma \ref{lemma:properties_Vn}. We first state Lemmas \ref{lemma:properties_Vn} and \ref{lemma:connect_to_ln}.

\begin{lemma}\label{lemma:properties_Vn}
Let $V_j, j=0,1,\ldots,$ be as in \eqref{eq:voltages_distflow}. Then,
\begin{enumerate}
\item $V_j\geq jk+1$,
\item $V_{j+1}-V_j = \sum_{i=0}^j \frac{k}{V_i} \to\infty$ as $j\to\infty$,
\item $V_{j+1}-V_j \leq k+\ln\left(1+jk\right)$,
\item $\frac{V_{j+1}-V_j}{V_j}=\mathcal{O}\left(\frac{\ln(j)}{j} \right)$.
\end{enumerate}
\end{lemma}
 
\begin{lemma}\label{lemma:connect_to_ln}
Let $V_j,j=0,1,\ldots,N-1$ as in \eqref{eq:voltages_distflow}. Then,
\begin{align*}
\frac{V_{j+1}-V_{j-1}}{V_j} = \ln(V_{j+1})-\ln(V_{j-1})+\mathcal{O}\left(\left(\frac{\ln(j)}{j} \right)^3 \right).
\end{align*}
\end{lemma} 

Both Lemmas \ref{lemma:properties_Vn} and \ref{lemma:connect_to_ln} are proven later in this section. Here, we discuss the efficacy of Lemma \ref{lemma:connect_to_ln} by numerical validation. 
We approximate,
\begin{align}
\frac{V_{j+1}-V_{j-1}}{V_j} = \left(\frac{V_{j+1}}{V_j}-1\right)+\left(1-\frac{V_{j-1}}{V_j} \right)\label{eq:approx_diff}
\end{align} by
\begin{align}
\approx & \ln\left(1+\left(\frac{V_{j+1}}{V_j}-1 \right) \right)-\ln\left(1-\left(1-\frac{V_{j-1}}{V_j} \right)\right)\nonumber\\
= & \ln\left(\frac{V_{j+1}}{V_{j}} \right)-\ln\left(\frac{V_{j-1}}{V_{j}} \right) = \ln(V_{j+1})-\ln(V_{j-1}).\label{eq:approx_ln}
\end{align} 
The efficacy of the approximation \eqref{eq:approx_ln} of \eqref{eq:approx_diff} is illustrated for the cases $k=0.001,\ k=0.01$ and $k=0.1$ in Figure \ref{fig:approximation_diff_and_log}. For these cases, the approximation already yields relative errors smaller than $0.5\%$ for $j\geq 10$.
\begin{figure}[h!]
\centering
\includegraphics[scale=0.5]{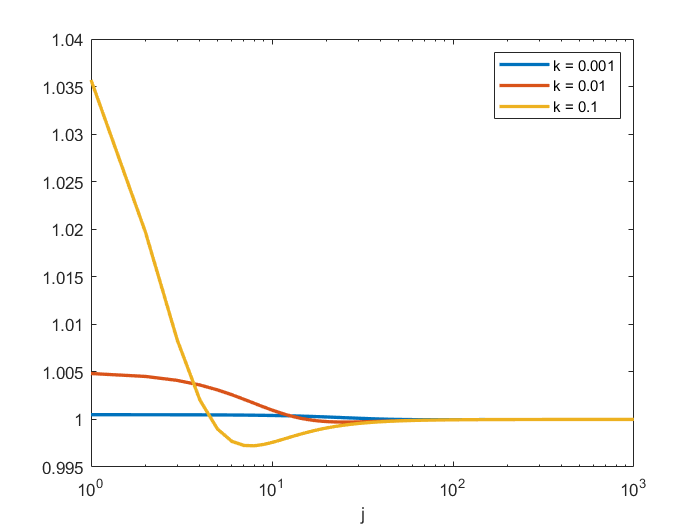}
\caption{Illustration of efficacy of the approximation \eqref{eq:approx_ln} of \eqref{eq:approx_diff} by showing the quotient of \eqref{eq:approx_ln} and \eqref{eq:approx_diff}, for three values of $k$.}
\label{fig:approximation_diff_and_log}
\end{figure}

Having Lemmas \ref{lemma:properties_Vn} and \ref{lemma:connect_to_ln} at our disposal, we are now ready to give the proof of Lemma \ref{lemma:lower_bound_V}.
\begin{proof}[Proof of Lemma \ref{lemma:lower_bound_V}]
In order to relate the recursion in \eqref{eq:voltages_distflow} to \eqref{eq:difference_W_j}, we write \eqref{eq:voltages_distflow} as
\begin{align}
(V_{j+1}-V_j)-(V_j-V_{j-1}) = \frac{k}{V_j},\quad j=1,2,\ldots,N-1\label{eq:Vn1}
\end{align} and multiply both sides of \eqref{eq:Vn1} by
\begin{align*}
V_{j+1}-V_{j-1} = (V_{j+1}-V_j)+(V_j-V_{j-1})
\end{align*} to obtain
\begin{align*}
(V_{j+1}-V_j)^2-(V_j-V_{j-1})^2 = k\frac{V_{j+1}-V_{j-1}}{V_j}.
\end{align*} Summing this over $j=1,2,\ldots,n$, we get
\begin{align}
(V_{n+1}-V_n)^2-(V_1-V_{0})^2 = k\sum_{j=1}^n \frac{V_{j+1}-V_{j-1}}{V_j}.\label{eq:connect_to_ln}
\end{align} 
We proceed with rewriting Equation \eqref{eq:connect_to_ln} to an expression that is similar to the one we obtained for the sequence $W_j, j=1,2,\ldots,N$ in Equation \eqref{eq:W_j+1-W_j} using Lemma \ref{lemma:connect_to_ln}. Then, we have
\begin{align}
(V_{n+1}-V_n)^2 & = (V_1-V_{0})^2+k\sum_{j=1}^n \frac{V_{j+1}-V_{j-1}}{V_j} \nonumber\\
& = (V_1-V_{0})^2+k\sum_{j=1}^n \left\{\ln(V_{j+1})-\ln(V_{j-1})+\mathcal{O}\left(\left(\frac{\ln(j)}{j}\right)^3\right)\right\}.\label{eq:rewrite_V_j}
\end{align} We observe a telescoping sum in the right-hand side of \eqref{eq:rewrite_V_j}, so we have
\begin{align*}
\sum_{j=1}^n \left\{\ln(V_{j+1})-\ln(V_{j-1})\right\} = \ln(V_{n+1})+\ln(V_n)-\left(\ln(V_1)+\ln(V_0)\right).
\end{align*} Furthermore, we introduce the following notation:
\begin{align*}
w^2(k) & = (V_1-V_{0})^2-k\left(\ln(V_1)+\ln(V_{0}) \right)\nonumber \\
& = k^2-k\ln(1+k)
\end{align*} and $R_j = \mathcal{O}\left(\left(\frac{\ln(j)}{j}\right)^3\right)$. Thus, we rewrite \eqref{eq:rewrite_V_j} to
\begin{align*}
(V_{n+1}-V_n)^2 & = w^2(k) + k(\ln(V_{n+1})+\ln(V_n))+\sum_{j=1}^n \mathcal{O}\left(\left(\frac{\ln(j)}{j}\right)^3\right) \\
& = w^2(k) + k(\ln(V_{n+1})+\ln(V_n))+\sum_{j=1}^n R_j.
\end{align*}
Recall that we want to derive a lower bound for the first order finite differences $V_{n+1}-V_n$. In order to do so, we use that $V_{n+1}\geq V_n$ (see \cite[Lemma 5.1]{Christianen2021}). Thus,
\begin{align*}
(V_{n+1}-V_n)^2 \geq w^2(k)+2k\ln(V_n)+\sum_{j=1}^n R_j.
\end{align*}
Since $\sum_{j=1}^{\infty}|R_j|<\infty$, we thus see that there is a constant $C$ such that
\begin{align*}
V_{n+1}-V_n \geq \left(C+2k\ln(V_n) \right)^{\frac{1}{2}},
\end{align*} as desired. 
\end{proof}
To complete the proof of Lemma \ref{lemma:lower_bound_V}, we are left to prove Lemmas \ref{lemma:properties_Vn} and \ref{lemma:connect_to_ln}. This is done in Sections \ref{subsubsec:proof_lemma_properties_Vn} and \ref{subsubsec:proof_lemma_connect_to_ln}, respectively.

\subsubsection{Proof of Lemma \ref{lemma:properties_Vn}}\label{subsubsec:proof_lemma_properties_Vn}

\begin{proof}[Proof of Lemma \ref{lemma:properties_Vn}]
The properties of the sequence $V_j, j=0,1,\ldots$ are given in the following way.
\begin{enumerate}
\item We have from \eqref{eq:voltages_distflow} for $j=1,2,\ldots$,
\begin{align}
V_{j+1}-V_j = V_j-V_{j-1}+\frac{k}{V_j}\geq V_j-V_{j-1}.\label{eq:Vn_summing}
\end{align} Hence, $V_{j+1}-V_j\geq V_1-V_0=(1+k)-1=k$ for $j=0,1,\ldots$. Then we get for $j=0,1,\ldots$
\begin{align*}
V_{j+1} = V_j+(V_{j+1}-V_j)\geq V_j+k,
\end{align*} and it follows from $V_0=1$ and induction that $V_j\geq 1+jk$ for $j=0,1,\ldots$.
\item We have from the identity in \eqref{eq:Vn_summing} by summation that
\begin{align*}
V_{j+1}-V_j & = (V_1-V_{0}) + \sum_{i=1}^j \frac{k}{V_i}\\
& = \frac{k}{V_0} + \sum_{i=1}^j \frac{k}{V_i}\\
& = \sum_{i=0}^j \frac{k}{V_i}.
\end{align*} When the latter expression would remain bounded by $B<\infty$ as $j\to\infty$, 
we would have $V_{j+1}\leq V_0 +jB, j=0,1,\ldots$. 
However, then $\sum_{i=0}^j \frac{k}{V_i}\geq \sum_{i=0}^{j-1}\frac{k}{V_0+iB}\to\infty$ as $j\to\infty$. Since this contradicts the assumption that the latter expression remains bounded, we must have that $\sum_{i=0}^j \frac{k}{V_i} \to \infty$ as $j\to\infty$.
\item Combining the results of items (1) and (2) gives us the desired result. Indeed,
\begin{align*}
V_{j+1}-V_j = \sum_{i=0}^j \frac{k}{V_i}\leq \sum_{i=0}^j \frac{k}{ik+1},
\end{align*} and
\begin{align*}
\sum_{i=0}^j \frac{k}{ik+1} & = k+\sum_{i=1}^j \frac{1}{i+1/k}\\
& \leq k + \int_{\frac{1}{2}+\frac{1}{k}}^{j+\frac{1}{2}+\frac{1}{k}}\frac{1}{x}dx \\
& = k+\left(\ln\left(j+\frac{1}{2}+\frac{1}{k}\right)-\ln\left(\frac{1}{2}+\frac{1}{k}\right) \right) \\
& = k+\ln\left(1+\frac{j}{\frac{1}{2}+\frac{1}{k}}\right)\\
& \leq k+\ln\left(1+jk\right).
\end{align*} 
\item This is a direct consequence of the inequalities in items (1) and (3). Combining (1) and (3) gives,
\begin{align*}
\frac{V_{j+1}-V_j}{V_j} \leq \frac{k+\ln(1+jk)}{1+jk}.
\end{align*} Hence, $\frac{V_{j+1}-V_j}{V_j} = \mathcal{O}\left(\frac{\ln(j)}{j} \right)$.
\end{enumerate}
\end{proof}

\subsubsection{Proof of Lemma \ref{lemma:connect_to_ln}}\label{subsubsec:proof_lemma_connect_to_ln}

\begin{proof}[Proof of Lemma \ref{lemma:connect_to_ln}]
We show the asymptotic behavior of $\frac{V_{j+1}-V_j}{V_j}$ as $j\to\infty$. Let, for $j=1,2,\ldots$,
\begin{align*}
X_j & = \frac{V_{j+1}}{V_j}-1 = \frac{V_{j+1}-V_j}{V_j}, \\
Y_j & = 1-\frac{V_{j-1}}{V_j} = \frac{V_j-V_{j-1}}{V_j}.
\end{align*} Then,
\begin{align}
0<X_j<1, 0<Y_j<1.\label{eq:bounds_X_Y}
\end{align} 
Indeed, from Lemma \ref{lemma:properties_Vn}, items 1 and 3,
\begin{align*}
X_j = \frac{V_{j+1}-V_j}{V_j}\leq 
\begin{cases}
& 1-\frac{1}{(k+1)^2}<1, \quad j=1, \\
& \frac{k}{1+jk}+\frac{\ln(1+jk)}{1+jk}\leq \frac{k}{1+jk} + \frac{1}{\exp(1)} \leq \frac{1}{2}+\frac{1}{\exp(1)}<1,\quad j=2,3,\ldots,N-1. 
\end{cases}
\end{align*} Here it has been used that the function $y^{-1}\ln(y), y\geq 1,$ has a global maximum at $y=\exp(1)$ that equals $\exp(-1)$. The other inequalities follow by the increasingness of the sequence $V_j,j=0,1,\ldots$ (see \cite[Lemma 5.1]{Christianen2021}). Furthermore, we have
\begin{align*}
X_j+Y_j & = \frac{V_{j+1}-V_{j-1}}{V_j},\\
X_j-Y_j & = \frac{V_{j+1}-2V_j+V_{j-1}}{V_j} = \frac{k}{(V_j)^2}>0.
\end{align*} Therefore,
\begin{align}
\ln(V_{j+1})-\ln(V_{j-1}) & = \ln\left(\frac{V_{j+1}}{V_j}\right)-\ln\left(\frac{V_{j-1}}{V_j}\right) \nonumber\\
& = \ln(1+X_j)-\ln(1-Y_j)\nonumber\\
& = \left(X_j-\frac{X_j^2}{2}+\frac{X_j^3}{3}-\ldots \right) - \left( -Y_j-\frac{Y_j^2}{2}-\frac{Y_j^3}{3}-\ldots \right) \nonumber\\
& = (X_j+Y_j)-\frac{1}{2}(X_j^2-Y_j^2)+\sum_{i=2}^{\infty} \frac{1}{i+1}\left((-1)^i X_j^{i+1}+Y_j^{i+1} \right)\nonumber\\
& = \frac{V_{j+1}-V_{j-1}}{V_j}-\frac{k(V_{j+1}-V_{j-1})}{2(V_j)^3}+\sum_{i=2}^{\infty} \frac{1}{i+1}\left((-1)^i X_j^{i+1}+Y_j^{i+1} \right)\label{eq:ln_differences}
\end{align} where the bounds in \eqref{eq:bounds_X_Y} assure convergence of the infinite series. Since $0<Y_j<X_j$, we have
\begin{align*}
\sum_{i=2}^{\infty} \frac{1}{i+1}\left|(-1)^iX_j^{i+1}+Y_j^{i+1}\right| &\leq \sum_{i=2}^{\infty} \frac{2}{i+1}X_j^{i+1}\\
& = X_j^3 \sum_{i=2}^{\infty} \frac{2}{i+1}X_j^{i-2}\\
& = \mathcal{O}\left(X_j^3 \right) = \mathcal{O}\left(\left(\frac{V_{j+1}-V_j}{V_j}\right)^3 \right).
\end{align*} Thus, we get that,
\begin{align*}
\frac{V_{j+1}-V_{j-1}}{V_j}& =\ln(V_{j+1})-\ln(V_{j-1})+\mathcal{O}\left(\frac{k(V_{j+1}-V_{j-1})}{2{V_j}^3} + \left(\frac{V_{j+1}-V_{j}}{V_j} \right)^3 \right)\\
& = \ln(V_{j+1})-\ln(V_{j-1})+\mathcal{O}\left(\left(\frac{\ln(j)}{j} \right)^3\right).
\end{align*} In the last line, we used Lemma \ref{lemma:properties_Vn}, item (4).
\end{proof}





\subsection{Proof of Lemma \ref{lemma:equivalence}}\label{subsec:proof_lemma_equivalence}
\begin{proof}[Proof of Lemma \ref{lemma:equivalence}]
We establish the (non-trivial) implication from (1) to (2). Assume there is $n\geq n_0(k)$ such that $V_n\geq W_n$. We claim that $V_j\geq W_j$ for all $j\geq n$. Indeed, when there is a $n_2>n$ such that $V_{n_2}<W_{n_2}$, we let $n_3:=\max\{j:n\leq j\leq n_2,V_j\geq W_j \}$. Then $V_{n_3}\geq W_{n_3}$ and $V_j<W_j$ for $n_3< j\leq n_2$. However, since $\psi(j)$ is strictly decreasing and $n_3+1>n_0(k)$, we have
\begin{align*}
V_{n_3+1}-V_{n_3} & \geq \left(C+2k\ln(V_{n_3}) \right)^{\frac{1}{2}}\\
& \geq \left(C-1+2k\ln(V_{n_3}) \right)^{\frac{1}{2}}\\
& \geq \left(\psi(n_0(k))+2k\ln(W_{n_3+1}) \right)^{\frac{1}{2}}\\
& \geq \left(\psi(n_3+1)+2k\ln(W_{n_3+1}) \right)^{\frac{1}{2}}\\
& \geq W_{n_3+1}-W_{n_3},
\end{align*} which implies $V_{n_3+1}\geq W_{n_3+1}$. This contradicts the definition of $n_3$. Since the choice of $n_2$ is arbitrary, we have that $V_j\geq W_j$ for all $j\geq n$. The implication from (2) to (1) is immediate.
\end{proof}

\subsection{Proof of Lemma \ref{lemma:V_j_geq_W_j}}\label{subsec:proof_lemma_V_j_geq_W_j}
\begin{proof}[Proof of Lemma \ref{lemma:V_j_geq_W_j}]
Let $n=2,3,\ldots$ and $j\geq n$. Then,
\begin{align}
V_j & = V_n + \sum_{i=n}^{j-1} (V_{i+1}-V_i)\nonumber\\
& = V_n + \sum_{i=n}^{j-1} \left(\sum_{l=n}^i \left[(V_{l+1}-V_l)-(V_l-V_{l-1})\right]+(V_n-V_{n-1}) \right)\nonumber\\
& = V_n + (j-n)(V_n-V_{n-1})+\sum_{i=n}^{j-1}\left(\sum_{l=n}^{i}(V_{l+1}-2V_l+V_{l-1}) \right) \nonumber\\
& = V_n + (j-n)(V_n-V_{n-1})+\sum_{i=n}^{j-1}\left(\sum_{l=n}^{i}\frac{k}{V_l} \right).\label{eq:V_j_expression}
\end{align} Now suppose that there is a $n=2,3,\ldots$ such that
\begin{align}
V_j \geq W_j,\quad  \text{for all}\  j\geq n.\label{eq:V_j_boundary}
\end{align} Then,
\begin{align}
\frac{1}{V_l}\leq \frac{1}{W_l} = \frac{1}{l(2k\ln(l))^{\frac{1}{2}}},\quad l=n,n+1,\ldots,\label{eq:W_l}
\end{align} and so by \eqref{eq:V_j_expression} for all $j\geq n$,
\begin{align}
V_j \leq V_n+(j-n)(V_n-V_{n-1})+\frac{k}{(2k)^{\frac{1}{2}}}\sum_{i=n}^{j-1}\left(\sum_{l=n}^{i} \frac{1}{l(\ln(l))^{\frac{1}{2}}} \right).\label{eq:V_j_expression_1}
\end{align} We use the Euler-Maclaurin formula in its simplest form: for $h\in C^2[n,\infty)$, we have
\begin{multline*}
\sum_{l=n}^i h(l) = \int_n^i h(x)dx + \frac{1}{2}(h(i)+h(n)) + \frac{1}{12}(h'(i)-h'(n))-\int_n^i h''(x)\frac{B_2(x-\floor*{x})}{2}dx,
\end{multline*} where $B_2(t) = (t-\frac{1}{2})^2-\frac{1}{12}$ is the Bernoulli polynomial of degree 2 that satisfies $|B_2(t)|\leq \frac{1}{6}, 0\leq t\leq 1$. Using this with $h(x) = \frac{1}{x(\ln(x))^{\frac{1}{2}}}, x\geq 2$, so that
\begin{align*}
h'(x) & = -\frac{1}{x^2(\ln(x))^{\frac{1}{2}}}-\frac{1}{2x^2(\ln(x))^{\frac{3}{2}}},\\
h''(x) & = \frac{2}{x^3(\ln(x))^{\frac{1}{2}}}+\frac{\frac{3}{2}}{x^3(\ln(x))^{\frac{3}{2}}}+\frac{\frac{3}{4}}{x^3(\ln(x))^{\frac{5}{2}}},\\
\int_n^i h(x) dx & = \int_n^i \frac{1}{x(\ln(x))^{\frac{1}{2}}}dx = 2(\ln(i))^{\frac{1}{2}}-2(\ln(n))^{\frac{1}{2}},
\end{align*} we get
\begin{align}
\sum_{l=n}^{i} \frac{1}{l(\ln(l))^{\frac{1}{2}}} = 2(\ln(i))^{\frac{1}{2}}-2(\ln(n))^{\frac{1}{2}}+\mathcal{O}\left(\frac{1}{n(\ln(n))^{\frac{1}{2}}} \right).\label{eq:euler_maclaurin_1}
\end{align} Next, from the Euler-Maclaurin formula with $h(x)=(\ln(x))^{\frac{1}{2}}$, we have
\begin{align}
\sum_{i=n}^{j-1}(\ln(i))^{\frac{1}{2}} = \int_n^{j-1}(\ln(x))^{\frac{1}{2}}dx + \mathcal{O}\left((\ln(j))^{\frac{1}{2}} \right),\label{eq:euler_maclaurin_2}
\end{align} and obviously
\begin{align}
\sum_{i=n}^{j-1}\left((\ln(n))^{\frac{1}{2}}+\mathcal{O}\left(\frac{1}{n(\ln(n))^{\frac{1}{2}}} \right) \right) = (j-n)\left((\ln(n))^{\frac{1}{2}}+\mathcal{O}\left(\frac{1}{n(\ln(n))^{\frac{1}{2}}} \right) \right).\label{eq:V_n_15}
\end{align} Thus, from \eqref{eq:euler_maclaurin_1} and \eqref{eq:euler_maclaurin_2}, we can write the right-hand side of \eqref{eq:V_j_expression_1} as
\begin{multline*}
 V_n+(j-n)(V_n-V_{n-1})+\\
\frac{k}{(2k)^{\frac{1}{2}}}\left(2\int_n^{j-1}(\ln(x))^{\frac{1}{2}}dx + \mathcal{O}\left((\ln(j))^{\frac{1}{2}} \right)+(j-n)\left((\ln(n))^{\frac{1}{2}}+\mathcal{O}\left(\frac{1}{n(\ln(n))^{\frac{1}{2}}}\right) \right) \right),\end{multline*} which simplifies to
\begin{align}
(2k)^{\frac{1}{2}}\int_n^{j-1}(\ln(x))^{\frac{1}{2}}dx + \mathcal{O}(j).\label{eq:V_n_16}
\end{align} Next, we use the substitution $u := (\ln(x))^{\frac{1}{2}}$ and partial integration, to obtain
\begin{align*}
\int_n^{j-1} (\ln(x))^{\frac{1}{2}}dx & = \int_{(\ln(n))^{\frac{1}{2}}}^{(\ln(j-1))^{\frac{1}{2}}} u\cdot 2u\exp(u^2)du \\
& = \left[u\exp(u^2)\right]_{{(\ln(n))^{\frac{1}{2}}}}^{{(\ln(j-1))^{\frac{1}{2}}}}-\int_{{(\ln(n))^{\frac{1}{2}}}}^{{(\ln(j-1))^{\frac{1}{2}}}} \exp(u^2)du \\
& = (j-1)(\ln(j-1))^{\frac{1}{2}}-n(\ln(n))^{\frac{1}{2}}-\int_{(\ln(n))^{\frac{1}{2}}}^{(\ln(j-1))^{\frac{1}{2}}} \exp(u^2)du.
\end{align*}
Using the second elementary inequality \eqref{eq:inequalities_int_exp} in Lemma \ref{lemma:ineq_I(y)}, we conclude that
\begin{align*}
\int_n^{j-1} (\ln(u))^{\frac{1}{2}}du = j(\ln(j))^{\frac{1}{2}}+\mathcal{O}\left(\frac{j}{(\ln(j))^{\frac{1}{2}}}\right),\quad j\to\infty.
\end{align*}
It thus follows from \eqref{eq:V_j_expression_1}, \eqref{eq:V_n_15} and \eqref{eq:V_n_16} that
\begin{align}
V_j \leq (2k)^{\frac{1}{2}}j(\ln(j))^{\frac{1}{2}}+\mathcal{O}(j).\label{eq:V_j_inequality_order_j} 
\end{align} Hence, from \eqref{eq:W_l} and \eqref{eq:V_j_inequality_order_j},
\begin{align*}
V_j 
& = W_j\left(1+\mathcal{O}\left(\frac{1}{(2k\ln(j))^{\frac{1}{2}}}\right) \right)\\
& = W_j\left(1+\smallO(1)\right),\quad j\to\infty.
\end{align*} In a similar fashion, if there is an $n=2,3,\ldots$ such that
\begin{align*}
V_j\leq j(2k\ln(j))^{\frac{1}{2}},\quad \text{for all}\ j\geq n,
\end{align*} then
\begin{align*}
V_j &\geq j(2k\ln(j))^{\frac{1}{2}}+\mathcal{O}(j),
\end{align*} which also yields
\begin{align*}
V_j &= W_j(1+\smallO(1)), \quad j\to\infty.
\end{align*}
\end{proof}

\section{Conclusion}\label{sec:conclusion}
Continuous and discrete Emden-Fowler type equations appear in many fields such as mathematical physics, astrophysics and chemistry, but also in electrical engineering, and more specifically under a popular power flow model. The specific Emden-Fowler equation we study, appears as a discrete recursion that governs the voltages on a line network and as a continuous approximation of these voltages.
We show that the asymptotic behavior of the solution of the continuous Emden-Fowler equation \eqref{eq:voltages_approx}, i.e. the approximation of the discrete recursion, and the asymptotic behavior of the solution of its discrete counterpart \eqref{eq:voltages_distflow}, are the same.

\section*{Acknowledgments}
This research is supported by the Dutch Research Council through the TOP programme under contract number 613.001.801.

\begin{appendices}
\section{Proofs for Section \ref{SUBSEC:ASSOCIATED_PROPERTIES}}\label{sec:existence_uniqueness_w}

\subsection{Proof of Lemma \ref{lemma:alternative_f}}
\begin{lemma}\label{lemma:alternative_f}
Let $f(t)$ be given by \eqref{eq:voltages_approx}. Then, we can alternatively write $f(t)$ by
\begin{align}
\frac{f'(t)}{(w^2+2k\ln(f(t)/y))^{\frac{1}{2}}} = 1,\quad t\geq 0,\label{eq:first_alternative_f}
\end{align} and
\begin{align}
\int_{(W^2-\ln(y))^{\frac{1}{2}}}^{(W^2+\ln(f(t))/y))^{\frac{1}{2}}}\exp(v^2)dv = \frac{t}{y}\sqrt{\frac{1}{2}k}\exp(W^2),\quad t\geq 0,\label{eq:second_alternative_f}
\end{align} where $W^2 := \frac{w^2}{2k}$.
\end{lemma}
\begin{proof}
From \eqref{eq:voltages_approx}, we get
\begin{align}
f'(u)f''(u) = kf'(u)/f(u),\quad 0\leq u\leq t.\label{eq:f4}
\end{align} Integrating Equation \eqref{eq:f4} over $u$ from 0 to $t$ using $f(0)=y, f'(0)=w$ we get
\begin{align*}
\int_0^t f'(u)f''(u)du = \frac{1}{2}(f'(t))^2 - \frac{1}{2}w^2 = \int_0^t \frac{kf'(u)}{f(u)} = k\ln(f(t)/y). 
\end{align*} Hence, for $t>0$,
\begin{align*}
\frac{f'(t)}{\left(w^2+2k\ln(f(t)/y)\right)^{\frac{1}{2}}} = 1,
\end{align*} as desired. Integrating $f'(u)/(w^2+2k\ln(f(t)/y))^{\frac{1}{2}})=1$ from $u=0$ to $u=t$, while substituting $s=f(u)\in [1,f(t)]$, we get
\begin{align}
\int_0^t \frac{\frac{df}{ds}\frac{ds}{du}}{(w^2+2k\ln(f(u)/y))^{\frac{1}{2}}} du & = \int_1^{f(t)} \frac{1}{(w^2+2k\ln(s/y))^{\frac{1}{2}}}ds = t.\label{eq:alternative_rep_f}
\end{align} By introduction of $W^2 = \frac{w^2}{2k}$, the expression becomes
\begin{align}
\frac{1}{\sqrt{2k}}\int_1^{f(t)} \frac{1}{(W^2+\ln(s/y))^{\frac{1}{2}}}ds = t.\label{eq:integral_f}
\end{align} Substituting $v=(W^2+\ln(s/y))^{\frac{1}{2}}, s=y\exp(v^2-W^2),ds=2s(W^2+\ln(s/y))^{\frac{1}{2}}dv$ in the integral \eqref{eq:integral_f}, we get
\begin{align*}
\int_{(W^2-\ln(y))^{\frac{1}{2}}}^{(W^2+\ln(f(t)/y))^{\frac{1}{2}})}\exp(v^2)dv = \frac{1}{2}\exp(W^2)\sqrt{2k}\frac{t}{y} = \frac{t}{y}\sqrt{\frac{1}{2}k}\exp(W^2),\quad t\geq 0,
\end{align*} as desired. This concludes the proof.
\end{proof}

\subsection{Proof of Lemma \ref{lemma:existence_uniqueness_w}}
\begin{lemma}\label{lemma:existence_uniqueness_w}
Let $k>0$. There exists a unique $w\geq 0$ such that the solution of $f(t)f''(t)=k,\ t\geq 0; f(0)=1,f'(0)=w$ satisfies $f(1)=1+k$.
\end{lemma}
\begin{proof}[Proof]
Again, we rely on the representation of $f$ in \eqref{eq:f(1)_condition_2}. Thus the condition $f(1)=1+k$ can be written as
\begin{align}
\int_1^{1+k}\frac{1}{(w^2+2k\ln(s))^{\frac{1}{2}}}ds = 1.\label{eq:decrease_in_w}
\end{align} The left-hand side of \eqref{eq:decrease_in_w} decreases in $w\geq 0$ from a value greater than $\sqrt{2}$ to 0 as $w$ increases from $w=0$ to $w=\infty$. Indeed, as to $w=0$ we consider 
\begin{align*}
F(k) = \int_1^{1+k}\frac{1}{(\ln(s))^{\frac{1}{2}}}ds, k\geq 0.
\end{align*} Then $F(0)=0$ and $F'(k_1) = (\ln(1+k_1))^{-\frac{1}{2}}>(k_1)^{-\frac{1}{2}},\quad k_1>0$, since $0<\ln(1+k_1)<k_1$ for $k_1>0$. Hence,
\begin{align*}
F(k)=F(0)+\int_0^k F'(k_1)dk_1 > \int_0^k \frac{1}{\sqrt{k_1}}dk_1 = 2\sqrt{k},\quad k>0.
\end{align*} This implies that
\begin{align*}
\int_1^{1+k} \frac{1}{\sqrt{2k\ln(s)}}ds = \frac{F(k)}{\sqrt{2k}}>\sqrt{2},\quad k>0.
\end{align*} That the left-hand side of \eqref{eq:decrease_in_w} decreases strictly in $w\geq 0$, to the value 0 at $w=\infty$, is obvious. We conclude that for any $k>0$ there is a unique $w>0$ such that \eqref{eq:decrease_in_w} holds.
\end{proof}

\subsection{Proof of Theorem \ref{thm:cases_k}}

\begin{proof}
From the definition of $F$ in \eqref{eq:F(t,k)}, it follows that $F(t,k)=0$ if and only if $f(t)=g(t)$. Furthermore, we have
\begin{align}
f(t)\geq g(t), 1\leq t < \infty \iff \max_{t\geq 1} F(t,k)\leq 0\label{eq:equivalence_maximum}.
\end{align} By Lemma \ref{lemma:F(t,k)}, we have, for any $k$, $\max_{t\geq 1}F(t,k) = F(t_0(k),k)$ and by Lemma \ref{lemma:F(t,k)_decreasing}, we have that $F(t_0(k),k)$ is a strictly decreasing function of $k$. Notice that, by \eqref{eq:F(t,k)}, we can alternatively write,
\begin{align*}
F(t_0(k),k) = \int_{(W^2+\ln(f(t_0(k))))^{\frac{1}{2}}}^{(W^2+\ln(g(t_0(k))))^{\frac{1}{2}}}\exp(v^2)dv.
\end{align*} Thus, by Lemma \ref{lemma:positive_small_k}, we have on the one hand, for small $k$, that $F(t_0(k),k)>0$, and by Lemma \ref{lemma:negative_large_k}, we have on the other hand, for large $k$, that $F(t_0(k),k)\leq 0$. Therefore, we conclude that $F(t_0(k),k)\leq 0$ is equivalent to $k\geq k_c$.
\end{proof}

\subsection{Proof of Lemma \ref{lemma:F(t,k)}}
\begin{lemma}\label{lemma:F(t,k)} Let $F(t,k)$ be given as in \eqref{eq:F(t,k)}. Then, for any $k$,
\begin{align*}
\max_{t\geq 1} F(t,k) = F(t_0(k),k),
\end{align*} where $t_0(k)$ is given by \eqref{eq:equation_psi}.
\end{lemma}

\begin{proof}
To find, for a given $k>0$, the maximum of $F(t,k)$ over $t\geq 1$, we compute from \eqref{eq:F(t,k)} 
\begin{align}
\frac{\partial F}{\partial t}(t,k) & = -\sqrt{\frac{k}{2}}\exp(W^2)+\frac{d}{dt}\left((W^2+\ln(g(t)))^{\frac{1}{2}}\right)\exp(W^2+\ln(g(t)))\nonumber\\
& = \frac{1}{2}(W^2+\ln(g(t)))^{-\frac{1}{2}}\frac{g'(t)}{g(t)}\exp(W^2+\ln(g(t)))-\sqrt{\frac{k}{2}}\exp(W^2)\nonumber\\
& = \exp(W^2)\left(\frac{1}{2}(W^2+\ln(g(t)))^{-\frac{1}{2}}\frac{g'(t)}{g(t)}\exp(\ln(g(t)))-\sqrt{\frac{k}{2}} \right)\nonumber\\
& = \exp(W^2)\sqrt{\frac{k}{2}}\left(\left(\sqrt{\frac{1}{2k}}\sqrt{\frac{g'(t)^2}{W^2+\ln(g(t))}} \right)-1 \right)\nonumber\\
& = \exp(W^2)\sqrt{\frac{k}{2}}\left(\left(\sqrt{\frac{\left(\frac{g'(t)}{\sqrt{2k}}\right)^2}{W^2+\ln(g(t))}} \right)-1 \right).\label{eq:intermediate_partial_Ft}
\end{align} Then, using \eqref{eq:diff_g(t)} in \eqref{eq:intermediate_partial_Ft}, we get
\begin{align}
\frac{\partial F}{\partial t}(t,k)& = \exp(W^2)\sqrt{\frac{k}{2}}\left(\left(\sqrt{\frac{\frac{\psi(t)}{2k}+\ln(g(t))}{W^2+\ln(g(t))}} \right)-1 \right).\label{eq:F(t,k)alternative}
\end{align} Then, $\frac{\partial F}{\partial t}(t,k) = 0$ if and only if $\frac{\psi(t)}{2k} = W^2$ or in other words, if and only if $\psi(t)=w^2$. Recall from \eqref{eq:equation_psi} that the unique solution $t>1$ of the equation $\psi(t)=w^2$ is given by $t_0(k)$. Thus, we have
\begin{align*}
\frac{\partial F}{\partial t}(t_0(k),k) = 0.
\end{align*} Since $\psi(t)$ is strictly decreasing in $t>1$, while $W^2$ does not depend on $t$, we have from \eqref{eq:F(t,k)alternative} that $\frac{\partial^2F}{\partial t^2}(t_0(k),k)<0$. Hence, for $k>0$,
\begin{align*}
\max_{t\geq 1} F(t,k) = F(t_0(k),k),
\end{align*} which completes the proof.
\end{proof}

\subsection{Proof of Lemma \ref{lemma:F(t,k)_decreasing}}
\begin{lemma}\label{lemma:F(t,k)_decreasing}
Let $F(t,k)$ be given as in \eqref{eq:F(t,k)}. Then, $F(t_0(k),k)$ is a strictly decreasing function of $k$, i.e.,
\begin{align*}
\frac{\partial F}{\partial k}(t_0(k),k)<0, \quad k>0.
\end{align*}
\end{lemma}

\begin{proof}
We compute $\frac{\partial F}{\partial k}(t,k)$ for any $t>1$, and set $t=t_0(k)$ in the resulting expression. Thus, from \eqref{eq:F(t,k)},
\begin{multline*}
\frac{\partial F}{\partial k}(t,k) = \frac{1}{2}(W^2+\ln(g(t)))^{-\frac{1}{2}}\frac{d}{dk}\left(W^2+\ln(\sqrt{k}t(2\ln(t))^{\frac{1}{2}}) \right)\exp(W^2+\ln(g(t)))-\\
-W'\exp(W^2)-\frac{t}{2\sqrt{2k}}\exp(W^2)-t\sqrt{\frac{1}{2}k}(W^2)'\exp(W^2).
\end{multline*} Simplifying this expression, yields
\begin{multline*}
\frac{\partial F}{\partial k}(t,k) = \exp(W^2)\left(\frac{1}{2}g(t)(W^2+\ln(g(t)))^{-\frac{1}{2}}\left((W^2)'+\frac{1}{2k}\right)-W'-\frac{t}{2\sqrt{2k}}-t\sqrt{\frac{1}{2}k}(W^2)'\right).
\end{multline*}
From $(W^2)'=2WW'$, we then have
\begin{align}
\frac{\partial F}{\partial k}(t,k) & =  \exp(W^2)\left(\frac{(WW'+\frac{1}{4k})g(t)}{(W^2+\ln(g(t)))^{\frac{1}{2}}}-W'-\frac{t}{2\sqrt{2k}}-t\sqrt{2k}WW' \right)\nonumber\\
& = \exp(W^2)\left(\left(\frac{g(t)}{(W^2+\ln(g(t)))^{\frac{1}{2}}}-t\sqrt{2k} \right)(WW'+\frac{1}{4k})-W' \right).\label{eq:partial_F_2}
\end{align} We next take $t=t_0(k)$ in \eqref{eq:partial_F_2}, so that we can use that
\begin{align*}
g'(t_0(k))=(w^2+2k\ln(g(t_0(k))))^{\frac{1}{2}}
\end{align*} and $W^2 = \frac{w^2}{2k}$, and observe that
\begin{align*}
\frac{g(t)}{(W^2+\ln(g(t)))^{\frac{1}{2}}}-t\sqrt{2k} & = \sqrt{2k}\left(\frac{g(t)}{(w^2+2k\ln(g(t)))^{\frac{1}{2}}}-t \right)\\
& = \sqrt{2k}\left(\frac{g(t)}{g'(t)}-t \right) = \frac{-t\sqrt{2k}}{1+2\ln(t)},\quad t = t_0(k).
\end{align*} We claim that,
\begin{align*}
\frac{\partial F}{\partial k}(t_0(k),k) & = -\exp(W^2)\left(\frac{t\sqrt{2k}}{1+2\ln(t)}(WW'+\frac{1}{4k}+W') \right),\quad t = t_0(k),
\end{align*} is negative since $W(k)$ increases in $k>0$, strictly. The latter fact is proven in Lemma \ref{lemma:increasing_W}. We conclude that $F(t_0(k),k)$ is a strictly decreasing function of $k>0$.
\end{proof}

\begin{lemma}\label{lemma:increasing_W}
Let $f(t)$ be given by \eqref{eq:f} with initial conditions $f(0)=1$ and $f'(0)=w$, where $w$ is such that $f(1)=1+k$. Furthermore, let $W(k)=\frac{w}{\sqrt{2k}}$. Then, $W(k)$ is a strictly increasing function of $k$.
\end{lemma}
\begin{proof}
First, by Equation \eqref{eq:alternative_rep_f} with $y=1$, we get
\begin{align}
\int_1^{f(t)} \frac{1}{(w^2+2k\ln(s))^{\frac{1}{2}}}ds = t.\label{eq:third_alternative_f}
\end{align} Second, from the fundamental theorem of calculus, we have
\begin{align}
f(1) = 1+w+\int_0^1\left(\int_0^s \frac{kdu}{f(u)} \right)ds. \label{eq:f(1)_condition}
\end{align} Now, we derive the desired monotonicity property. We require $f(1)=1+k$. We get from \eqref{eq:f(1)_condition},
\begin{align}
\frac{w}{k} = 1-\int_0^1\left(\int_0^s \frac{du}{f(u)} \right) ds.\label{eq:w_frac_k}
\end{align} From, \eqref{eq:third_alternative_f}, with $t=1$ and $f(1)=1+k$, we get
\begin{align}
\int_{1}^{1+k} \frac{1}{(w^2+2k\ln(s))^{\frac{1}{2}}}ds = 1.\label{eq:f(1)_condition_2}
\end{align} From \eqref{eq:f(1)_condition_2}, noting that $w=w(k)$, we get then
\begin{align}
0 & = \frac{d}{dk}\left(\int_1^{1+k} \frac{1}{(w^2(k)+2k\ln(s))^{\frac{1}{2}}}ds \right)\nonumber\\
& = \frac{1}{(w^2(k)+2k\ln(1+k))^{\frac{1}{2}}}+\int_1^{1+k} -\frac{1}{2}\frac{2w(k)w'(k)+2\ln(s)}{(w^2(k)+2k\ln(s))^{\frac{3}{2}}}ds \nonumber\\
& = \frac{1}{(w^2(k)+2k\ln(1+k))^{\frac{1}{2}}}-w(k)w'(k)\int_1^{1+k} \frac{1}{(w^2(k)+2k\ln(s))^{\frac{3}{2}}}ds - \nonumber \\
& \quad \quad \quad \quad \quad \quad \quad \quad \quad \quad \quad \quad - \int_1^{1+k}\frac{\ln(s)}{(w^2(k)+2k\ln(s))^{\frac{3}{2}}}ds.\label{eq:diff_calculus}
\end{align} Hence, rewriting \eqref{eq:diff_calculus} yields,
\begin{multline}
w(k)w'(k)\int_1^{1+k} \frac{1}{(w^2(k)+2k\ln(s))^{\frac{3}{2}}}ds = \frac{1}{(w^2(k)+2k\ln(1+k))^{\frac{1}{2}}} \\
- \int_1^{1+k} \frac{\ln(s)}{(w^2(k)+2k\ln(s))^{\frac{3}{2}}}ds.\label{eq:wkw_primek}
\end{multline} Consider the last term in \eqref{eq:wkw_primek}. We have for $1\leq s\leq 1+k$,
\begin{align}
\frac{\ln(s)}{(w^2(k)+2k\ln(s))^{\frac{3}{2}}} & = \frac{\ln(s)}{w^2(k)+2k\ln(s)}\frac{1}{(w^2(k)+2k\ln(s))^{\frac{1}{2}}} \nonumber\\
& \leq \frac{\ln(1+k)}{w^2(k)+2k\ln(1+k)}\frac{1}{(w^2(k)+2k\ln(s))^{\frac{1}{2}}}.\label{ineq:lns}
\end{align} Therefore, by integrating over the inequality in \eqref{ineq:lns}, we get
\begin{align*}
\int_1^{1+k} \frac{\ln(s)}{(w^2(k)+2k\ln(s))^{\frac{3}{2}}}ds & \leq \frac{\ln(1+k)}{w^2(k)+2k\ln(1+k)}\int_1^{1+k}\frac{1}{(w^2(k)+2k\ln(s))^{\frac{1}{2}}}ds \nonumber\\
& = \frac{\ln(1+k)}{w^2(k)+2k\ln(1+k)},
\end{align*} where we used \eqref{eq:f(1)_condition_2}. Therefore, see \eqref{eq:wkw_primek},
\begin{align}
w(k)w'(k)\int_1^{1+k} \frac{1}{(w^2(k)+2k\ln(s))^{\frac{3}{2}}}ds \geq 
\frac{1}{(w^2(k)+2k\ln(1+k))^{\frac{1}{2}}}-\frac{\ln(1+k)}{w^2(k)+2k\ln(1+k)}>0,\label{eq:wkw_prime_2}
\end{align} where the latter inequality follows from
\begin{align*}
(w^2(k)+2k\ln(1+k))^{\frac{1}{2}} > (2k\ln(1+k))^{\frac{1}{2}} > \ln(1+k),
\end{align*} since $u>\ln(1+u)$ for $u>0$. We conclude from \eqref{eq:wkw_prime_2} that $w(k)$ strictly increases in $k>0$.

Next, we consider for a fixed $t>0$ the identity \eqref{eq:alternative_rep_f}. For any $s>1$, the integrand $\left(w^2(k)+2k\ln(s)\right)^{-\frac{1}{2}}$ decreases strictly in $k>0$, and hence $f(t)=f(t;k)$ increases strictly in $k>0$, since $t>0$ is fixed. As a consequence, we conclude from \eqref{eq:w_frac_k} that $w(k)/k$ strictly increases in $k>0$ since $1/f(u)$ strictly decreases in $k>0$ for any $u\in(0,1)$.
\end{proof}

\subsection{Proof of Lemma \ref{lemma:positive_small_k}}
\begin{lemma}\label{lemma:positive_small_k}
Let $F(t,k)$ be given as in \eqref{eq:F(t,k)}. Then, for small $k$, we have that $F(t_0(k),k)>0$. 
\end{lemma}
\begin{proof}
We have for $t>0$,
\begin{align*}
f(t)& =f(0)+tf'(0)+\frac{1}{2}t^2f''(\xi_t)\\
& = 1+tw+\frac{1}{2}t^2\frac{k}{f(\xi_t)},
\end{align*} where $\xi_t$ is a number between $0$ and $t$. Since $f(1)=1+k$ and $f(\xi_t)\geq 1 > 0$, it follows that $w\leq k$. Therefore,
\begin{align*}
f\left(\frac{1}{\sqrt{k}}\right) \leq 1 + \frac{w}{\sqrt{k}} + \frac{1}{2k}k \leq \frac{3}{2}+\sqrt{k}.
\end{align*} On the other hand
\begin{align*}
g\left(\frac{1}{\sqrt{k}}\right) = \frac{1}{\sqrt{k}}\left(2k\ln(\frac{1}{\sqrt{k}}) \right)^{\frac{1}{2}} = (-\ln(k))^{\frac{1}{2}},
\end{align*} and this exceeds $\frac{3}{2}+\sqrt{k}$ when $k$ is small enough. Numerically, by solving the equation $(-\ln(k))^{\frac{1}{2}} = \frac{3}{2}+\sqrt{k}$ for $k>0$, we find that $k<0.05$ is small enough. We conclude from \eqref{eq:F(t,k)} that $F(t_0(k),k)\geq F\left(\frac{1}{\sqrt{k}},k\right)>0$ when $k$ is small.
\end{proof}

\subsection{Proof of Lemma \ref{lemma:negative_large_k}}
\begin{lemma}\label{lemma:negative_large_k}
Let $F(t,k)$ be given as in \eqref{eq:F(t,k)}. Then, for large $k$, we have that $F(t_0(k),k)\leq 0$.
\end{lemma}

\begin{proof}
We show that $f(t)\geq g(t)$ for all $t\geq 1$ when $k$ is large enough. We have $f(1)=1+k>0=g(1)$. Now suppose that there is a $t>1$ such that $f(t)<g(t)$. Then there is also a $t_1>1$ such that $f(t_1)=g(t_1)$ and $f'(t_1)\leq g'(t_1)$. We infer, by the derivatives of the functions $f$ and $g$ given in Equations \eqref{eq:first_alternative_f} and \eqref{eq:diff_g(t)}, with \eqref{eq:psi}, from $f(t_1)=g(t_1)$ and $f'(t_1)\leq g'(t_1)$, that
\begin{align}
w^2 \leq 2k+\frac{k}{2\ln(t)}-k\ln(2k\ln(t)) = \psi(t) \quad \text{at}\ t=t_1.\label{eq:rhs_psi}
\end{align} At the same time, we have by convexity of $f(t), 0\leq t<\infty$, and $f(1)=1+k$ that
\begin{align*}
f(t)\geq 1+tk,\quad t\geq 1.
\end{align*} Hence, when $\frac{1}{2}k\geq \ln(t)$, we have
\begin{align*}
f(t)\geq 1+tk> tk = t\left(2k\cdot\frac{1}{2}k\right)^{\frac{1}{2}} \geq t(2k\ln(t))^{\frac{1}{2}} = g(t).
\end{align*} Since $f(t_1)=g(t_1)$, we thus have that $t_1>\exp(\frac{1}{2}k)$. The right-hand side of \eqref{eq:rhs_psi} decreases in $t>1$, since the function $\psi(t)$ is strictly decreasing, and its value at $t=t_1$ is therefore less than
\begin{align*}
2k+\frac{k}{2\cdot \frac{1}{2}k}-k\ln(2k\cdot\frac{1}{2}k) = 2k+1-2k\ln(k).
\end{align*} Since $2k+1-2k\ln(k)<0$ for large $k$, \eqref{eq:rhs_psi} cannot hold for large $k$. Numerically, by solving the equation $2k+1-2k\ln(k)=0$, for $k>0$, we find that $k>3.2$ is large enough. This gives the result.
\end{proof}

\subsection{Proof of Theorem \ref{thm:bounds_f/g}}
\begin{proof} Let $f(t)=f(t;k)$ be given by \eqref{eq:f} with initial conditions $f(0)=1,f'(0)=w$ such that $f(1)=1+k$, and let $g(t)=g(t;k)$ be given by \eqref{eq:f(t)_approx}. Before we turn to the proof of inequalities \eqref{eq:lowerbound_f_g} and \eqref{eq:upperbound_f_g}, we first state some numerical results obtained by Newton's method: the unique number $k_c$ that determines whether the ratio of $f$ and $g$ is positive or not, is given by $k_c=1.0384$, the corresponding value of $w$ such that $f(1)=1+k_c$ is given by $w(k_c)=0.6218$ and the corresponding solution to the equation $\psi(t_0(k_c)) = w(k_c)^2$ is given by $t_0(k_c)=t_2(k_c)=18.3798$. Furthermore, by Newton's method, we have that
\begin{align*}
\frac{f_0(x)}{g(x;1)}\leq 1, x_1\leq x\leq x_2;\quad \frac{f_0(x)}{g(x;1)}>1, 1\leq x<x_1\ \text{or}\ x>x_2,
\end{align*} where $x_1 = 2.4556$ and $x_2 = 263.0304$, and $g(x;k)=(2k)^{\frac{1}{2}}x(\ln(x))^{\frac{1}{2}}$. Additionally, the minimum of the ratio $f_0(x)$ and $g(x;1)$ is given by
\begin{align}
\min_{x\geq 1} \frac{f_0(x)}{g(x;1)} = \min_{x_1\leq x\leq x_2} \frac{f_0(x)}{x(2\ln(x))^{\frac{1}{2}}} \approx 0.8829,\label{eq:minimum_ratio_f_g}
\end{align} and is attained at $x_{\text{min}} = 5.7889$. The maximum of the ratio $f_0(x)$ and $g(x;1)$ is given by
\begin{align*}
\max_{x\geq 1} \frac{f_0(x)}{g(x;1)} \approx 1.0223,
\end{align*} and is attained at $x_{\text{max}} = 380223$. However, the computation of the maximum of the ratio of $f_0(x)$ and $g(x;1)$ is much more involved than the computation of the minimum of the ratio of $f_0(x)$ and $g(x;1)$, because evaluation of the function $f_0(x)$ for large entries is difficult. In Lemma \ref{lemma:maximum_ratio_f0_g}, we content ourselves with a reasonably sharp upper bound on the maximum of the ratio of $f_0(x)$ and $g(x;1)$ over $x\geq x_2$.

We now turn to the proof of inequality \eqref{eq:lowerbound_f_g}. We consider two regimes, i.e., $t_1(k)\leq t \leq \sqrt{2/k}$ and $t\geq \sqrt{2/k}$. We have for $t_1(k)\leq t\leq \sqrt{2/k}$,
\begin{align}
\frac{f(t;k)}{g(t;k)} & \geq \frac{1}{t(2k\ln(t))^{\frac{1}{2}}} \nonumber \\
& \geq \frac{1}{\sqrt{2/k}\left(2k\ln(\sqrt{2/k}) \right)^{\frac{1}{2}}} \nonumber \\
& = \frac{1}{2(\ln(\sqrt{2/k}))^{\frac{1}{2}}},\label{eq:first_regime}
\end{align} where we used that $f(t;k)$, with $f(0;k)=1$ and $g(t;k)=(2k)^{\frac{1}{2}}t(\ln(t))^{\frac{1}{2}}$ are positive, increasing functions of $t>1$. Next, we let $t\geq \sqrt{2/k}$. We have
\begin{align}
\frac{f(t;k)}{g(t;k)} & = \frac{cf_0(a+bt)}{t\sqrt{2k\ln(t)}} \nonumber \\
& = \frac{f_0(a+bt)}{g(a+bt;1)}\frac{ac+bct}{t\sqrt{k}}\sqrt{\frac{\ln(a+bt)}{\ln(t)}}. \label{eq:three_terms}
\end{align} We consider each factor in the right-hand side of \eqref{eq:three_terms} separately. For the first factor, we use the numerical result that the minimum of the ratio of the functions $f_0$ and $g$ is given in \eqref{eq:minimum_ratio_f_g}. For the second and third factor, we notice, from \eqref{eq:a}--\eqref{eq:c} and $t\geq \sqrt{2/k}$, that
\begin{align}
a>0, b>\sqrt{k}, bc=\sqrt{k}, a+bt\geq \sqrt{k}\sqrt{2/k} = \sqrt{2}.\label{eq:two_inequalities_a_b}
\end{align} Hence, for the second factor, we get
\begin{align*}
\frac{ac+bct}{t\sqrt{k}} > \frac{\sqrt{k}t}{t\sqrt{k}} = 1,
\end{align*} and for the third factor,
\begin{align}
\min_{t\geq \sqrt{2/k}}\frac{\ln(a+bt)}{\ln(t)} > \min_{t\geq \sqrt{2/k}} \frac{\ln(t\sqrt{k})}{\ln(t)}. \label{min_frac_of_ln}
\end{align} However, the right-hand side of \eqref{min_frac_of_ln} is equal to $1$ when $k\geq 1$, and equal to $\frac{\ln{\sqrt{2}}}{\ln\sqrt{2/k}}$ when $0<k<1$. Therefore, 
\begin{align*}
\min_{t\geq \sqrt{2/k}} \frac{\ln(t\sqrt{k})}{\ln(t)} \geq \frac{\ln(\sqrt{2/k_c})}{\ln(\sqrt{2/k})}
\end{align*} when $0<k\leq k_c$. Hence, combining the inequalities for each factor in \eqref{eq:three_terms}, we get, for $t\geq \sqrt{2/k}$,
\begin{align*}
\frac{f(t;k)}{g(t;k)} \geq 0.8829\cdot 1\cdot \left(\frac{\ln\left(\sqrt{2/k_c}\right)}{\ln\left(\sqrt{2/k}\right)}\right)^{\frac{1}{2}} = \frac{0.5055}{\left(\ln(\sqrt{2/k})\right)^{\frac{1}{2}}}.
\end{align*} Together with \eqref{eq:first_regime} this gives the desired result.

Now, we turn to the proof of inequality \eqref{eq:upperbound_f_g}. We follow the same approach as in the proof of inequality \eqref{eq:lowerbound_f_g}. Thus, we consider each factor of the right-hand side of \eqref{eq:three_terms} separately. The right-hand side of \eqref{eq:three_terms} is now to be considered for $t\geq t_2(k)$, and so it is important to have specific information about $t_2(k)$. We claim that $t_2(k)\geq \exp(e^{2-k_c}/2k)$. This claim is proven in Lemma \ref{lemma:t2(k)} below. 

We consider $x=a+bt$ with $t\geq t_2(k)$. Now, by the first two inequalities in \eqref{eq:two_inequalities_a_b}, we get
\begin{align*}
a+bt_2(k)> \sqrt{k}\exp(e^{2-k_c}/2k).
\end{align*} Notice that the function $k>0 \mapsto \sqrt{k}\exp(e^{2-k_c}/2k)$ is a decreasing function of $k$ for $0<k\leq k_c$, and therefore,
\begin{align*}
a+bt_2(k) > \sqrt{k_c}\exp(e^{2-k_c}/2k_c) \approx 3.5909>x_1.
\end{align*} Hence, it is sufficient to bound the function $f_0(x)/g(x;1)$ for $x\geq x_2$. Furthermore, by Lemma \ref{lemma:maximum_ratio_f0_g}, we have
\begin{align*}
\frac{f_0(x)}{g(x;1)} &\leq 1.12.
\end{align*}
For the second factor, we notice, since $bc=\sqrt{k}$, that we have
\begin{align*}
\frac{ac+bct}{t\sqrt{k}} = 1+\frac{ac}{t\sqrt{k}}.
\end{align*} Now, from \eqref{eq:a} and \eqref{eq:c},
\begin{align*}
\frac{ac}{\sqrt{k}} & = \frac{1}{\sqrt{k}}\sqrt{2}\int_0^{\frac{w}{\sqrt{2k}}}\exp(v^2)dv \cdot \exp(-w^2/2k) \\
& \leq \frac{1}{\sqrt{k}}\sqrt{2}\cdot \frac{w}{\sqrt{2k}} = \frac{w}{k}\\
& \leq \frac{w(k_c)}{k_c} \approx 0.5988,
\end{align*} where in the last line it has been used that $w/k$ is an increasing function of $k$; see Lemma \ref{lemma:increasing_W}. Hence, for all $k\leq k_c$ and all $t\geq t_2(k)\geq t_0(k)$, we can bound the second factor by
\begin{align*}
\frac{ac+bct}{t\sqrt{k}}\leq 1+\frac{w(k_c)}{k_ct_0(k_c)} \approx 1.0326.
\end{align*} For the third factor, we have by \eqref{eq:a},
\begin{align*}
a = \sqrt{2}\int_0^{\frac{w}{\sqrt{2k}}}\exp(v^2)dv \leq \frac{w}{\sqrt{k}}\exp(w^2/2k).
\end{align*} Hence, by \eqref{eq:b},
\begin{align*}
\ln(a+bt) &\leq \frac{w^2}{2k} + \ln\left(\left(\frac{w}{k}+t\right)\sqrt{k} \right) \\
&\leq \frac{w^2(k_c)}{2k_c}+\ln\left(\left(\frac{w(k_c)}{k_c}+t\right)\sqrt{k} \right).
\end{align*} Therefore, for all $t\geq t_2(k)$,
\begin{align*}
\left(\frac{\ln(a+bt)}{\ln(t)} \right)^{\frac{1}{2}} & \leq \left(\frac{\frac{w^2(k_c)}{2k_c}+\ln\left((\frac{w(k_c)}{k_c}+t)\sqrt{k} \right)}{\ln(t)} \right) \\
& \leq \left(1+\frac{\frac{w^2(k_c)}{2k_c}+\ln(\sqrt{k})+\frac{w(k_c)}{tk_c}}{\ln(t)} \right)^{\frac{1}{2}} \\
& \leq \left(1+\frac{\frac{w^2(k_c)}{2k_c}+\ln(\sqrt{k_c})+\frac{w(k_c)}{t_0(k_c)k_c}}{\ln(t_0(k_c))} \right)^{\frac{1}{2}} \approx 1.0400.
\end{align*} Combining all inequalities for each factor in \eqref{eq:three_terms}, we get
\begin{align*}
\frac{f(t;k)}{g(t;k)} \leq 1.12 \cdot 1.0326\cdot 1.0400 \approx 1.2023.
\end{align*}
\end{proof}

\begin{lemma}\label{lemma:maximum_ratio_f0_g}
Let $f_0(x)$ be given by \eqref{eq:f_0(x)} and let $g(x;1)$ be given by $g(x;1)=x(2\ln(x))^{\frac{1}{2}}$. Then,
\begin{align*}
\frac{f_0(x)}{g(x;1)}\leq 1.12,\quad x\geq x_2.
\end{align*}
\end{lemma} 

\begin{proof}
From \eqref{eq:f_0(x)},\eqref{eq:Ux} and the first inequality of \eqref{eq:inequalities_int_exp}, we have for $x>0$
\begin{align*}
\frac{x}{\sqrt{2}} = \int_0^{(\ln(f_0(x)))^{\frac{1}{2}}}\exp(u^2)du \geq \frac{f_0(x)-1}{2(\ln(f_0(x)))^{\frac{1}{2}}},
\end{align*} i.e.,
\begin{align}
f_0(x)\leq 1+x(2\ln(f_0(x))^{\frac{1}{2}}.\label{eq:ineq_f0(x)}
\end{align} Let $x>0$ be fixed and consider the mapping
\begin{align*}
G:z\geq 1 \mapsto 1+x(2\ln(z))^{\frac{1}{2}}.
\end{align*} Then $G$ maps $[1,\infty)$ onto $[1,\infty)$ with $G(1)=1$ and $G(\infty)=\infty$, $G$ is strictly concave on $[1,\infty)$, and $G'(z)$ decreases from $\infty$ to $0$ as $z$ increases from $1$ to $\infty$. Therefore, $G$ has a unique fixed point $z(x)$ in $(1,\infty)$. We have for any $z_1>1,z_2>1$ that
\begin{align}
z_1\leq z(x) \iff z_1 \leq 1+x(2\ln(z_1))^{\frac{1}{2}},\quad z_2\geq z(x) \iff z_2\geq 1+x(2\ln(z_2))^{\frac{1}{2}}. \label{eq:z1,z2}
\end{align} Note that $f_0(x)\leq z(x)$ by \eqref{eq:ineq_f0(x)}. Now let $\alpha>1$ and consider $z_2 = 1+\alpha g(x;1)$, where we take $\alpha$ such that $z_2\geq 1+x(2\ln(z_2))^{\frac{1}{2}}$. An easy computation shows that with this $z_2 = 1+\alpha g(x;1)$,
\begin{align}
z_2\geq 1+x(2\ln(z_2))^{\frac{1}{2}} \iff g(x;1)\leq \frac{x^{\alpha^2}-1}{\alpha}.\label{eq:z2,g}
\end{align} We consider all this for $x\geq x_2 = 263.0340$. We have
\begin{align*}
g(x_2;1)=x_2(2\ln(x_2))^{\frac{1}{2}} = \frac{x_2^{\alpha^2}-1}{\alpha}
\end{align*} for $\alpha = 1.1115 :=\alpha_2$. Furthermore, when we have an $x\geq x_2$ such that $g(x;1)\leq \frac{1}{\alpha}(x^{\alpha^2}-1)$, we have
\begin{align*}
\frac{d}{dx}\left[\frac{1}{\alpha}(x^{\alpha^2}-1) \right] & = \frac{\alpha}{x}x^{\alpha^2}\geq \frac{\alpha}{x}\left(1+\alpha g(x;1) \right) \\
& = \frac{\alpha}{x}+\frac{\alpha^2}{x}x(2\ln(x))^{\frac{1}{2}} \\
& > \alpha^2(2\ln(x))^{\frac{1}{2}} \\
& \geq (2\ln(x))^{\frac{1}{2}} + \frac{1}{(2\ln(x))^{\frac{1}{2}}} = g'(x;1),
\end{align*} where the last inequality holds when $\alpha^2-1\geq \frac{1}{2\ln(x)}$. The latter inequality certainly holds for $\alpha = \alpha_2$ and $x\geq x_2$. Hence,
\begin{align*}
g(x_2;1) = \frac{x_2^{\alpha_2^2}-1}{\alpha_2};\quad g'(x;1)<\frac{d}{dx}\left[\frac{1}{\alpha_2}(x^{\alpha_2^2}-1) \right],\quad x\geq x_2,
\end{align*} and we conclude that
\begin{align*}
g(x;1)\leq \frac{x^{\alpha_2^2}-1}{\alpha_2},\quad x\geq x_2.
\end{align*} Then, by \eqref{eq:z1,z2} and \eqref{eq:z2,g} and $f_0(x)\leq z(x)$, we get that
\begin{align*}
z_2 = 1+\alpha_2g(x;1)\geq z(x)\geq f(x), \quad x\geq x_2.
\end{align*} This implies that
\begin{align*}
\frac{f_0(x)}{g(x;1)} \leq \alpha_2 + \frac{1}{g(x;1)}\leq \alpha_2 + \frac{1}{g(x_2;1)} \leq 1.12, \quad x\geq x_2,
\end{align*} since $\alpha_2\leq 1.112$ and $(g(x_2;1))^{-1}\leq 0.008$ as required. 
\end{proof}

\begin{lemma}\label{lemma:t2(k)}
Let $f(t)$ be given by \eqref{eq:f} with initial conditions $f(0)=1, f'(0)=w$ such that $f(1)=1+k$, and let $g(t)$ be given by \eqref{eq:f(t)_approx} for $0<k\leq k_c$. Then,
\begin{align*}
t_2(k) \geq \exp\left(\frac{e^{2-k_c}}{2k} \right),
\end{align*} where $t_2(k)$ is given as in Theorem \ref{thm:cases_k}, case $b$.
\end{lemma}
\begin{proof}
Let $f(t)$ be given by \eqref{eq:f} with initial conditions $f(0)=1, f'(0)=w$ such that $f(1)=1+k$, and let $g(t)$ be given by \eqref{eq:f(t)_approx}. By Theorem \ref{thm:cases_k} case (b), we have $t_2(k)\geq t_0(k)$, where $t_0(k)$ is the unique root of the equation
\begin{align}
2+\frac{1}{2\ln(t)}-\ln(2k\ln(t)) = \frac{w^2}{k},\label{eq:soly}
\end{align} see \eqref{eq:equation_psi}. If we denote $y=2\ln(t_0(k))$, then the solution $y(k)$ of \eqref{eq:soly} satisfies
\begin{align*}
y(k) > \frac{1}{k}\exp\left(2-\frac{w^2}{k}\right)\geq \frac{1}{k}\exp(2-\frac{w^2(k_c)}{k_c}) \geq \frac{1}{k}\exp(2-k_c),
\end{align*} since $\frac{w^2}{k}$ increases in $k$ and $w^2\leq k^2$. Now, using that $y=2\ln(t_0(k))$, we get
\begin{align*}
t_0(k) \geq \exp\left(\frac{e^{2-k_c}}{2k}\right).
\end{align*} Since $t_2(k)\geq t_0(k)$, we have the desired result.
\end{proof}
\end{appendices}

\bibliography{PhD-Asymptotic_Analysis}
\end{document}